\documentclass[11pt]{article}
\usepackage{amsmath, amsthm,amssymb,latexsym,mathabx}
\usepackage{caption,subcaption,xcolor}
\usepackage{multirow,array,kotex}
\usepackage{dsfont}
\usepackage{bbold}
\usepackage{booktabs}
\usepackage{graphicx,cite,graphics}
\usepackage{enumerate,tikz}
\numberwithin{equation}{section}
\newcommand{\bw}{\mathbf{w}}

\newcommand{\bx}{\mathbf{x}}
\usepackage{wasysym}

\newcommand{\CI}{
\begin{tikzpicture}
\draw  (0,0) -- (1.5ex,0) -- (0.75ex,1.5ex) --(0,0);
\end{tikzpicture}}

\newcommand{\CII}{
\begin{tikzpicture}
\draw  (0,0) -- (1.5ex,0) -- (1.5ex,1.5ex) --(0,1.5ex)--(0,0);
\end{tikzpicture}
}

\newcommand{\CIII}{
\begin{tikzpicture}
\draw  (0,0) -- (1.5ex,0) -- (1.5ex,1.5ex) --(0,1.5ex)--(0,0);
\draw (0,0)-- (1.5ex,1.5ex);
\end{tikzpicture}
}
\theoremstyle{plain}
\newtheorem{theorem}{Theorem}[section]
\newtheorem{proposition}[theorem]{Proposition}
\newtheorem{lemma}[theorem]{Lemma}
\newtheorem{corollary}[theorem]{Corollary}
\newtheorem{claim}[theorem]{Claim}
\newtheorem{observation}[theorem]{Observation}

\theoremstyle{definition}
\newtheorem{definition}[theorem]{Definition}
\newtheorem{algorithm}[theorem]{Algorithm}
\newtheorem{example}[theorem]{Example}

\usepackage{titling}
\begin{document}
\title{On toric ideals arising from signed graphs}
\author{JiSun Huh${}^{a}$, \  Sangwook Kim${}^{b}$  \ and Boram Park${}^{a}$.\\[1.5ex]
\small ${}^{a}$ Department of Mathematics, Ajou University, Suwon 16499, Republic of Korea\\
\small ${}^{b}$ Department of Mathematics, Chonnam National University, Gwangju 61186, Republic of Korea}

\maketitle
\date{}

\vspace{-0.5cm}
\begin{abstract}
\small
A \textit{signed graph} is a pair $(G,\tau)$ of a graph $G$ and its sign $\tau$, where
a \textit{sign} $\tau$ is a function from $\{ (e,v)\mid e\in E(G),v\in V(G), v\in e\}$ to $\{1,-1\}$. Note that graphs or digraphs are special cases of signed graphs.
In this paper, we study the toric ideal $I_{(G,\tau)}$ associated with a signed graph $(G,\tau)$, and the results of the paper give a unified idea to explain some known results on the toric ideals of a graph or a digraph.
We characterize all primitive binomials of $I_{(G,\tau)}$, and then
focus on the complete intersection property.
More precisely, we find a complete list of graphs $G$ such that $I_{(G,\tau)}$ is a complete intersection for every sign $\tau$.
\end{abstract}

\noindent\textbf{Keywords} Signed graph; Toric ideal; Primitive element; Complete intersection\\

\noindent\textbf{MSC2010}	14M25, 05C22,  05C25

\section{Introduction}
Throughout the paper, a graph means a finite simple graph. A finite graph allowed to have a multiple edge or a loop is called a multigraph. For a graph $G$, we set $V(G)=\{v_1,\ldots, v_n\}$, $E(G)=\{e_1,\ldots,e_m\}$ and $\mathbb{e}=(e_1,\ldots,e_m)$ unless otherwise specified.
For a positive integer $n$, we denote $\{1,\ldots,n\}$ by $[n]$. For an integer vector $\mathbb{b}$, $\mathbb{b}^+$ (resp. $\mathbb{b}^-$) means the vector whose $i$th entry is $\max\{b_i,0\}$ (resp. $-\min\{b_i,0\}$).
For an integer vector $\mathbb{x}=(x_1,\ldots,x_m)$,
$\mathbb{e}^{\mathbb{x}}$ means a monomial $e_1^{x_1}e_2^{x_2}\cdots e_m^{x_m}$.\footnote{Throughout the paper, to denote a vector, we use $\mathbb{a}$, $\mathbb{b}$, $\mathbb{c}$, etc. The standard bold type letters ($\mathbf{a}$, $\mathbf{b}$, $\mathbf{c}$, etc) are for walks in a graph.}

Let $K[e_1,\ldots,e_m]$ be a polynomial ring in $m$ variables over a field $K$. For  an $n\times m$ integer matrix $A$
 without zero columns,
the ideal
\[ I_A=\left< \mathbb{e}^{\mathbb{b}^+} -
\mathbb{e}^{\mathbb{b}^-} \in K[e_1,\ldots,e_m] \mid \mathbb{b}\in\mathbb{Z}^m \text{ and } A\mathbb{b}=\mathbb{0}  \right>\]
is called the \textit{toric ideal} associated with $A$. It is well-known that a toric ideal  is a prime binomial  ideal.
For more details about toric ideals and related topics, see \cite{S1996, Hibi}.

A (homogeneous) toric ideal not only defines a projective toric variety (see \cite{Fulton,S2002}), but also provides wide applications in other areas, such as algebraic statistics, dynamical system, hypergeometric differential equations, toric geometry, graph theory, and so on, see \cite{ES1996, S1996, ES2004}.
Toric ideals arising from various kinds of combinatorial objects have been widely studied by many researchers, see  \cite{kashiwabara2010toric, lason2014toric,PS2014,ohsugi2017grobner} for some recent results.
Especially, the toric ideal of a graph or a digraph, which is the toric ideal associated with its vertex-edge incidence matrix, has been an interesting topic (see \cite{ohsugi1999koszul,ohsugi1999toric, RTT2012minimal,BGR2015,biermann2017bounds,gap2015,GRV2013CI,reyes2005complete,gitler2017cio}).

A major line of research on toric ideal arising from a combinatorial object focuses on a `special' set of binomials of the ideal (giving a combinatorial interpretation). Among them, the set of primitive binomials, which is known to form the Graver basis, was studied widely related to a problem initiated by Sturmfels, called true degree problem. (See \cite{S1996,TT2015,T2011universal,R2017bound} for detail.)
For a toric ideal $I_A$,
an irreducible binomial $B=B^+-B^-$ of $I_A$ is \textit{primitive} if there exists no other binomial  $B_0=B_0^+-B_0^-$ such that  $B_0^+|B^+$ and
$B_0^-|B^-$.
For the toric ideal of a graph, the primitive binomials and some other important binomials were characterized in \cite{RTT2012minimal}.
The primitive binomials of the toric ideal of a digraph are nicely stated in \cite{gitler2010ring,GRV2013CI}. See Subsection~\ref{subsubsec:primitive} for the primitive binomials of the toric ideal of a graph/digraph.

Another important  research direction on a toric ideal is about the complete intersection property.
A toric ideal $I_A$ associated with an $n\times m$ integer matrix $A$  has the height $\textrm{ht}(I_ A ) = m- \textrm{rank}(A)$.  We say $I_A$ is a \textit{complete intersection} if it is generated by $\textrm{ht}(I_A)$ elements (see \cite{S1996}).
A complete intersection toric ideal was first studied by Herzog in \cite{H1970}, and it is known that the Hilbert series of the corresponding quotient ring $R/I$ can be computed easily when $I$ is a complete intersection.
The complete intersection property of the toric ideal from a combinatorial object was also investigated by many researchers, see \cite{BG2015,MT2005,reyes2005complete,BGR2015,GRV2013CI,gitler2017cio,gitler2010ring,TT2013,M2019}. We summarize some known results on the toric ideals of graphs/digraphs in Subsection~\ref{subsubsec:CI}.

In this paper, we consider toric ideals of signed graphs, as a generalization of graphs and digraphs.
An \textit{incidence} of a graph $G$ is a pair $(e,v)$ of an edge $e$ and a vertex $v$ such that $v$ is an endpoint of $e$.
A \textit{sign} $\tau$ of $G$ is a function from the set of all incidences to the set $\{1,-1\}$, and a \textit{signed graph} is a pair $(G,\tau)$ of a graph $G$ and its sign $\tau$.
For a signed graph $(G,\tau)$ with $n$ vertices and $m$ edges, the \textit{incidence matrix} $A(G,\tau)$ of $(G,\tau)$ is an $n\times m$ matrix whose rows are labeled by the vertices $v_1, \ldots, v_n$ and columns are labeled by the edges $e_1, \ldots, e_m$ such that $[A(G,\tau)]_{ij}=\tau(e_j,v_i)$ if $v_i$ is incident to $e_j$, and $[A(G,\tau)]_{ij}=0$ otherwise. With an abuse of notation, we often consider
the codomain of $\tau$ is $\{+,-\}$. See Figure~\ref{fig:signed} for an example.
\begin{figure}[ht!]
  \centering
  \includegraphics[width=10.5cm,page=1]{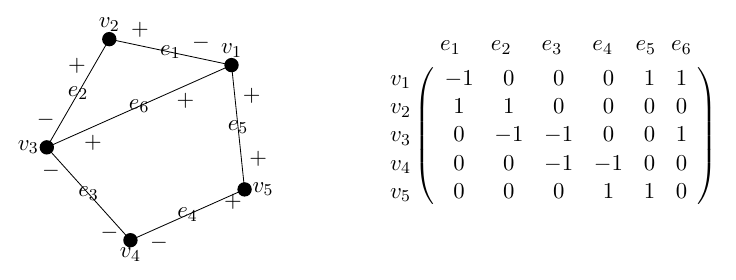}\\
  \caption{A signed graph $(G,\tau)$ and its incidence matrix $A(G,\tau)$}\label{fig:signed}
\end{figure}
We remark that if a sign $\tau$ is a constant function, then $(G,\tau)$ is just a graph. If a sign $\tau$ satisfies that $\tau(e,u) \tau(e,v)=-1$ for each edge $e=uv$, then $(G,\tau)$ is equal to a digraph. The \textit{toric ideal of a signed graph} $(G,\tau)$, denoted by $I_{(G,\tau)}$, is the toric ideal associated with the incidence matrix $A(G,\tau)$ in the polynomial ring $K[e_1,\ldots,e_m]$ over a field $K$.

As long as the authors are aware, the toric ideal of a signed graph is firstly considered in this paper, and so we start our research from a fundamental question on generators. We completely characterize the primitive binomials of $I_{(G,\tau)}$ of a signed graph $(G,\tau)$ with graph theory language. This gives a way to explain the previous results on graphs/digraphs in a unified idea.
The latter part of the paper sheds light on the complete intersection property  of  $I_{(G,\tau)}$. We give a necessary and sufficient condition for a  graph $G$ to have a complete intersection $I_{(G,\tau)}$ for every sign $\tau$, see Theorems~\ref{thm:characterization:CI}~and~\ref{thm:cis:general:graph}.
We emphasize that this result is more than unifying the results of graphs/digraphs in \cite{GRV2013CI, BGR2015}, since there are infinitely many graphs $G$ such that the toric ideal of $G$ and its every orientation are complete intersections but $I_{(G,\tau)}$ is not a complete intersection for some sign $\tau$ (see Section~3).
 Lastly, we find a full list of such graphs without assuming $2$-connectedness, see Theorem~\ref{thm:characterization:CI}.

\section{Preliminaries}\label{sec:prel}
This section gives some basic notion and terminology in graphs, and then summarizes some known results on toric ideals of graphs and digraphs. In addition, we explain how to define binomials from walks in a signed graph, which generate $I_{(G,\tau)}$.

\subsection{Basic notion for walks in a graph }\label{subsec:basic:graph:walks}
For a graph $G$, let $\mathbf{w}: v_{i_1}e_{j_1} \cdots e_{j_t}v_{i_{t+1}}$ be a walk or $(v_{i_1},v_{i_{t+1}})$-walk,
which is an alternating sequence  of vertices $v_i$'s and edges $e_j$'s where $e_{j_{\ell}}=v_{i_{\ell}}v_{i_{\ell+1}}$ for each $\ell \in[t]$.
We call $v_{i_{\ell}}$ (resp. $e_{j_{\ell}}$)  the $\ell$th vertex (resp. edge) term of $\bw$.
A vertex term is said to be \textit{internal} if it is neither first nor last.
We let $V(\bw)$ be the set of vertex terms of $\bw$ and $E(\bw)$ be the multiset of the edge terms in $\bw$. We denote the multigraph with the vertex  set $V(\bw)$ and the edge set $E(\bw)$ by $[\bw]$. The underlying simple graph of $[\bw]$ is a subgraph of $G$, but $[\bw]$ may not be a subgraph of $G$ by multiple edges.
See Figure~\ref{fig:walk}.

\begin{figure}[ht!]
  \centering
  \includegraphics[width=14cm,page=2]{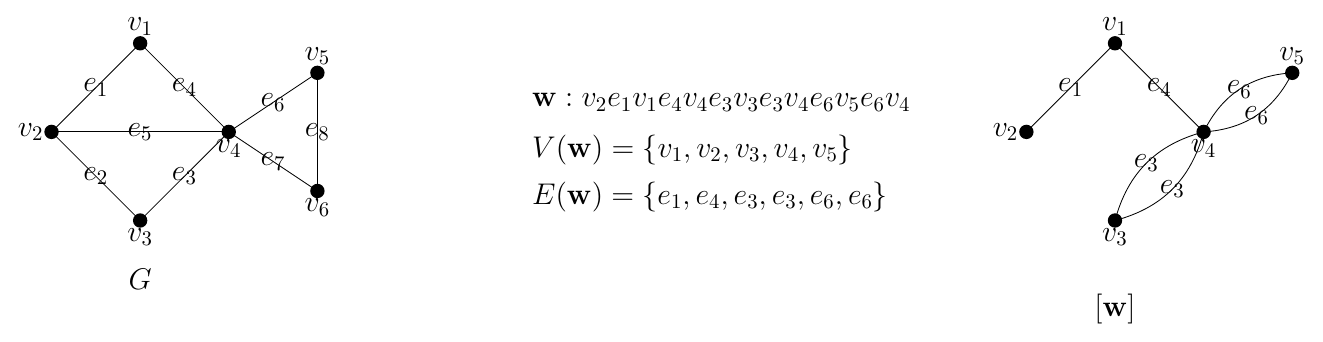}\\\caption{A walk $\bw$ in a graph $G$ and $V(\bw)$, $E(\bw)$, and $[\bw]$}\label{fig:walk}
\end{figure}

The \textit{length}  of $\bw$ is the number of edge terms in $\bw$.
A \textit{subwalk} of $\bw$ is a subsequence of $\bw$ which is a walk, and
a \textit{section} is a subwalk consisting of consecutive terms of $\bw$.
When we consider a subwalk or a section of a closed walk $\bw$, the terms are considered cyclically so that the last edge term is consecutive to the first vertex term.

For two walks $\bw$ and $\bw'$ in a graph, if the last vertex term of $\bw$ and the first vertex term of $\bw'$ are equal, then we denote by $\bw+\bw'$ the walk going through $\bw$ and then $\bw'$.
If $\bw_0, \ldots, \bw_k$ are sections of a walk $\bw$ such that
 $\bw=\bw_0+\cdots+\bw_k$, then we call this form a \textit{section-decomposition} of $\bw$.
If every $\bw_i$ is a nontrivial walk, then we say it is nontrivial.

The walk obtained by reading a walk $\bw$ in the reverse order is denoted by $\bw^{-1}$.
For a closed walk $\bw: v_{i_1}e_{j_1}\cdots e_{j_t}v_{i_1}$ $(t\ge2)$, we say a vertex $v\in V(\bw)$ is \textit{repeated} if $v$ appears in $v_{i_1}e_{j_1}v_{i_2}\cdots v_{i_t}$ ($e_{j_t}v_{i_1}$ is deleted from $\bw$) at least two times.
Note that if a closed walk $\bw$ has a repeated vertex $v\in V(\bw)$,
then $\bw$ has a nontrivial section-decomposition $\bw_0+\bw_1$ for some closed walks $\bw_0$ and $\bw_1$ (whose first vertex terms are $v$).

\subsection{Toric ideals of graphs and digraphs}\label{subsec:known:graphs:digraphs}

 Recall that the \textit{toric ideal} $I_G$ of a graph (resp. digraph) $G$  is the toric ideal  $I_{(G,\tau)}$ when $\tau$ is a constant function (resp. $\tau(e,u)\tau(e,v)=-1$ for every edge $e=uv$).
In this section we summarize some previous results on the toric ideals of graphs/digraphs, which will be used in this paper.

\subsubsection{Primitive binomials}\label{subsubsec:primitive}
Let $A$ be an $n\times m$ matrix without zero columns.
An irreducible binomial of $I_A$ has a form of $\mathbb{e}^{\mathbb{b}^+}-\mathbb{e}^{\mathbb{b}^-}$ for some  $\mathbb{b}\in\mathbb{Z}^{m}$ with $A\mathbb{b}=\mathbb{0}$.
An irreducible binomial $\mathbb{e}^{\mathbb{b}^+}-\mathbb{e}^{\mathbb{b}^-}$ is called \textit{primitive} if there exists no other binomial
$\mathbb{e}^{\mathbb{c}^+}-\mathbb{e}^{\mathbb{c}^-}$ such that
$\mathbb{e}^{\mathbb{c}^+}| \mathbb{e}^{\mathbb{b}^+}$ and
$\mathbb{e}^{\mathbb{c}^-}|\mathbb{e}^{\mathbb{b}^-}$.

For a closed nontrivial walk $\bw: v_{i_1}e_{j_1}\cdots v_{i_{2t}}e_{j_{2t}}v_{i_1}$ of even length in a graph $G$,  let $B_{\bw}=B^+-B^-$, where $B^+=e_{j_1}e_{j_3}\cdots e_{j_{2t-1}}$ and $B^-=e_{j_2}e_{j_4}\cdots e_{j_{2t}}$.
Here, the same closed walk can be written in different ways but associated binomials differ only in the sign.
It is  observed that (see \cite{Ree}) $I_G$ is generated by those binomials $B_{\bw}$. A necessary condition for the primitive binomials was firstly studied in \cite{ohsugi1999toric} and a necessary and sufficient condition was established in \cite{RTT2012minimal} as follows. When two graphs $G$ and $G'$  contain cliques $K$ and $K'$ of size $k$, respectively,
a graph obtained from $G$ and $G'$ by identifying  $K$ and $K'$ is called a $k$-\textit{clique sum} of $G$ and $G'$.

\begin{theorem}
[{\cite[Theorem 3.2]{RTT2012minimal}}]\label{thm:graph:primitive}
For a closed walk $\bw$ of even length in a graph, the binomial $B_\bw$ is primitive if and only if the following hold:
\begin{itemize}
\item[\rm(i)]
The multigraph $[\bw]$ is constructed by $1$-clique sums of cycles of length at least two such that  every vertex of $[\bw]$ belongs to at most two cycles.
\item[\rm(ii)]
For every nontrivial section-decomposition $\bw=\bw_0+\bw_1$ into two closed walks $\bw_0$ and $\bw_1$, the length of each $\bw_i$ is odd.
\end{itemize}
\end{theorem}

In \cite{RTT2012minimal},  other important sets of binomials in $I_G$ were also characterized and we omit them here as it is not related to our main purposes.

Let $G$ be a graph and $D$ be its orientation. The primitive binomials of $I_D$ are much more simply described. For every cycle $\bw$
of $G$, we define $B_{\bw}=B^+-B^-$, where $B^+$ is the product of the clockwise oriented edges  and  $B^-$ is the product of the other edges.

\begin{theorem}
[{\cite[Proposition 2]{GRV2013CI}}]\label{thm:digraph:primitive}
For a graph $G$, let $D$ be its orientation.
The primitive binomials of $I_D$ are binomials $B_{\bw}$ associated with cycles $\bw$ of $G$.
\end{theorem}

\subsubsection{The complete intersection property}
\label{subsubsec:CI}
Recall that the toric ideal $I_A$ is a \textit{complete intersection} if it can be generated by  $\mathrm{ht}(I)$ elements, where  $\mathrm{ht}(I)$ is the \textit{height} of $I$.  It also holds that $\mathrm{ht}(I_A)=m-\mathrm{rank}(A)$.
For a connected graph $G$ with $n$ vertices and $m$ edges, $I_G$ is a complete intersection if and only if it is generated by $r(G)$ binomials, where
\[r(G)=\begin{cases}
m-n+1&\text{if }G\text{ is bipartite},\\
m-n&\text{otherwise. }
\end{cases}\]
For a disconnected graph $G$, $I_G$ is a \textit{complete intersection} if every connected component of $G$ has a complete intersection toric ideal.
Let $\mathcal{G}^{ci}$ be the set of all graphs with complete intersection toric ideals.
The bipartite graphs in $\mathcal{G}^{ci}$ are nicely characterized as follows.

\begin{theorem}
[{\cite[Corollary 3.4]{gitler2010ring}}]
For a connected bipartite graph $G$,  $G\in\mathcal{G}^{ci}$ if and only if $G$ is a ring graph. Here,
a \emph{ring graph} is a graph whose nonedge block is constructed by $2$-clique sums of cycles.
\end{theorem}

A pseudo-code to check whether a graph is in $\mathcal{G}^{ci}$ or not is given in \cite{BGR2015}, see Algorithm~\ref{thm:CI_Graphs}. The algorithm relies on inductive idea, which may not give an exact description on structures of graphs in $\mathcal{G}^{ci}$.
Instead, useful structural properties are provided in \cite{BGR2015,TT2013}, and some are listed in Theorem~\ref{graph:useful:facts}. For a graph $H$, the number of connected components which are bipartite is denoted by $b(H)$.

\begin{figure}[!ht]
\centering
\fbox{\parbox{0.9\textwidth}{\vspace{-0.2cm}
\begin{algorithm}[\cite{BGR2015}]\label{thm:CI_Graphs}
\end{algorithm}
\hspace{0.5cm} Input: a simple graph $G$

\hspace{0.5cm} Output: \textsc{True} if $G\in \mathcal{G}^{ci}$ or \textsc{False} otherwise

\bigskip

\hspace{0.5cm} $H:=G$; $\mathcal{B}:=\emptyset$

\hspace{0.5cm}  \textbf{while} $\exists v\in V(H)$ with $\deg_H(v)\le 2$ \textbf{do}

\hspace{1cm} \textbf{if} $\deg_H(v)=2$ and $b(H-v)=b(H)$ \textbf{then}

\hspace{1.5cm}  $W:=\{v\}\cup N_H(v)\cup \{u\in V(H)| b(H-\{u,v\})>b(H-u)\}$

\hspace{1.5cm}  \textbf{if} not exists a closed walk $\mathbf{w}$ of even length such that $V(\mathbf{w})=W$ \textbf{then}

\hspace{2cm}   \textbf{return} \textsc{False}

\hspace{1.5cm}   \textbf{end if}

\hspace{1.5cm}   Let $\mathbf{w}$  be a shortest  closed walk of even length such that $V(\mathbf{w})=W$.

\hspace{1.5cm}   $\mathcal{B}:=\mathcal{B}\cup \{B_{\mathbf{w}}\}$

\hspace{1cm}  \textbf{end if}

\hspace{1cm} $H:=H-v$

\hspace{0.5cm}  \textbf{end while}

\hspace{0.5cm} Let $H_1, \ldots, H_s$ be the connected components of $H$.

\hspace{0.5cm} \textbf{if} exists $i$ such that $H_i$ is neither odd band nor even M\"{o}bius band
 \textbf{then}

\hspace{1cm} \textbf{return} \textsc{False}

\hspace{0.5cm} \textbf{end if}

\hspace{0.5cm} Let $\mathcal{B}_i$ be a minimal generating set of $I_{H_i}$ fore each $1\le i\le s$.

\hspace{0.5cm} \textbf{if} $I_G$ is generated by $\mathcal{B}\cup \mathcal{B}_{1}\cup\cdots\cup\mathcal{B}_s$ \textbf{then}

\hspace{1cm} \textbf{return} \textsc{True}

\hspace{0.5cm} \textbf{end if}

\hspace{0.5cm} \textbf{return} \textsc{False}

\smallskip
}}\end{figure}

\begin{theorem}
[{\cite[Theorem 3.6, Corollary 3.9,  Lemma 6.2]{BGR2015}}, {\cite[Theorem 3.1, Corollary 5.6]{TT2013}}]
\label{graph:useful:facts}
Let $G$ be a connected graph in $\mathcal{G}^{ci}$. Then the following hold:
\begin{itemize}
\item[\rm (i)] If $G$ is not  bipartite, then $2|E(G)|\le 3|V(G)|-\sum_{v\in V(G)} b(G-v)$.
\item[\rm(ii)] $G$  has no $K_{2,3}$ as a subgraph.
\item[\rm(iii)] If $G$ is $2$-connected and has two cycles $C$ and $C'$ of odd length sharing exactly one vertex $v$, then there is an edge $e$ not incident to $v$ which connects $C$ and $C'$.
\item[\rm (iv)] If $G$ is $2$-connected and has disjoint two cycles $C$ and $C'$ of odd length, then there are two disjoint edges  $e_1$ and $e_2$ such that each $e_i$ connects $C$ and $C'$.
\item[\rm (v)] $G$ has at most two non-bipartite blocks.
\item[\rm(vi)] Every induced subgraph of $G$ belongs to $
\mathcal{G}^{ci}$.
\end{itemize}
\end{theorem}

In {\cite{BGR2015}}, 3-regular graphs with complete intersection toric ideals are characterized, see Theorem~\ref{thm:K4:ci}.
Instead of giving the definitions of bands or M\"{o}bius bands, we note that the complete graph $K_4$ is an even M\"{o}bius band. For the definitions, see \cite[Definition 4.2]{BGR2015}.

\begin{theorem}[{\cite[Theorem 4.4]{BGR2015}}]
\label{thm:K4:ci} For a 3-regular connected graph $G$,
$G\in \mathcal{G}^{ci}$ if and only if
it is an odd band or an even M\"{o}bius band.
\end{theorem}

The following lists results on digraphs.
For a connected graph $G$ with $n$ vertices and $m$ edges,
and an orientation $D$ of $G$, $I_D$ is a (binomial) complete intersection if and only if it is generated by $r(D)=m-n+1$ binomials. For an orientation $D$ of a disconnected graph, $I_D$ is a complete intersection if the toric ideal of the digraph restricted to every connected component of $G$ is a complete intersection.
Let $\mathcal{G}^{cio}$ be the set of graphs $G$ such that $I_D$ is generated a complete intersection for every orientation $D$ of $G$.

\begin{theorem}
[{\cite[Theorems~4 and 6, Corollary 4]{GRV2013CI}}]\label{thm:CI:digraph}
Let $G$ be a connected graph.
\begin{itemize}
\item[\rm(i)] $G\in \mathcal{G}^{cio}$ if and only if $G$ is constructed by clique sums of complete graphs and/or cycles.
\item[\rm(ii)] Every induced subgraph of a graph in $\mathcal{G}^{cio}$ belongs to $\mathcal{G}^{cio}$.
\end{itemize}
\end{theorem}

\subsection{Even-signed walks in $(G,\tau)$ and their associated binomials of $I_{(G,\tau)}$}\label{subsec:evenwalk}
In this subsection, we explain how to define binomials associated with closed walks in a signed graph, and this will play a key role in Section~\ref{sec:main}.

\begin{definition}\label{def:odd:even:walk}
Let $\bw:v_{i_1}e_{j_1}v_{i_2}\cdots e_{j_t}v_{i_{t+1}}$ be a walk in a signed graph $(G,\tau)$ of length at least two.
An internal vertex term $v_{i_{\ell}}$ of $\bw$ is \textit{unbalanced} if $\tau(e_{j_{\ell-1}},v_{i_{\ell}})\tau(e_{j_{\ell}},v_{i_{\ell}})=1$.
As long as $\bw$ is closed, we say $v_{i_1}$ (or $v_{i_{t+1}}$) is
\textit{unbalanced} if $\tau(e_{j_t},v_{i_1})\tau(e_{j_1},v_{i_1})=1$.
We define $\mu(\bw)=(-1)^{k}$, where $k$ is the number of unbalanced vertex terms of $\bw$.
We also say $\bw$ is \textit{even-signed} if $\mu(\bw)=1$, and $\bw$ is \textit{odd-signed} if $\mu(\bw)=-1$.
\end{definition}

Throughout the paper,
a walk/cycle with odd/even number of edge terms is said to be a walk/cycle of odd/even length.
A \textit{triangle}  means a cycle of length three.

A \textit{balanced section} $\bw_0$ of a walk $\bw$ is a maximal section of $\bw$ such that $\bw_0$ has no internal unbalanced vertex term.
If $\bw_0+\cdots+\bw_k$ is a section-decomposition of $\bw$ such that each $\bw_i$ is a balanced section, then this form is called a \textit{balanced section-decomposition} of $\bw$.

For a closed walk $\bw$, if it has no unbalanced vertex term, then it is even-signed and has exactly one balanced section which is itself. Otherwise, its balanced section-decomposition $\bw_0+\cdots+\bw_k$ is also unique up to cyclic permutations.
So, by choosing an unbalanced vertex as the first vertex term properly, we can denote by $\bw=\bw_0+\cdots+\bw_k$.
\begin{example}\label{ex2}
Consider a signed graph $(G,\tau)$ in Figure~\ref{fig:signed:sec3}, and its two closed walks $\bw$ and $\bw'$, where
\[\bw:v_3e_3v_4 e_4 v_5e_5v_1e_6v_3e_2v_2e_1v_1e_6v_3,\quad \bw':v_1e_1v_2e_2v_3e_6v_1e_5v_5e_5v_1.\]
\begin{figure}[ht!]
  \centering
  \includegraphics[width=11.5cm,page=3]{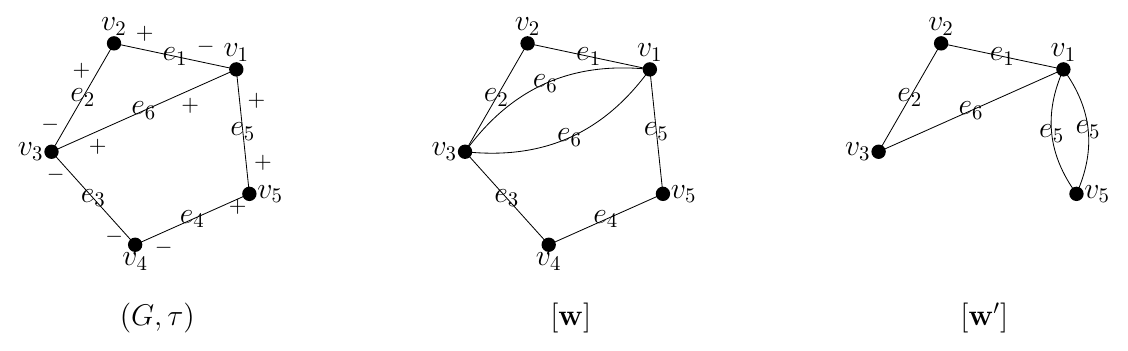}\\
  \caption{A signed graph $(G,\tau)$ and two closed walks $\bw$ and $\bw'$}\label{fig:signed:sec3}
\end{figure}

For a closed walk $\bw$, the 2nd, 3rd, 4th, and 6th vertex terms are the unbalanced vertex terms, which implies that $\bw$ is an even-signed walk in $(G,\tau)$ and the balanced sections are
$$ \bw_0: v_2e_1v_1e_6v_3e_3v_4,  \quad \bw_1: v_4 e_4 v_5, \quad \bw_2: v_5e_5v_1, \quad \bw_3: v_1e_6v_3e_2v_2. $$
The walk has a balanced section-decomposition $\bw_0+\bw_1+\bw_2+\bw_3$.
For the closed walk $\bw'$, there are three unbalanced vertex terms, the 2nd, 4th, and 5th vertex terms, which implies that $\bw'$ is odd-signed in $(G,\tau)$ and its balanced sections are
$$  \bw'_0: v_5e_5v_1{e_1}v_2, \quad  \bw'_1: v_2e_2v_3e_6v_1, \quad \bw'_2: v_1e_5v_5,$$ and
$\bw'_0+\bw'_1+\bw'_2$ is  a balanced section-decomposition of $\bw'$.
Note that  by  taking the first vertex term of $\bw$ or $\bw'$ properly, we can write $\bw=\bw_0+\bw_1+\bw_2+\bw_3$ or $\bw'=\bw'_0+\bw'_1+\bw'_2$.
\end{example}

\begin{definition}
Let $\bw$ be an even-signed closed walk in a signed graph, and $\bw_0+ \bw_1+\cdots+ \bw_{2k-1}$ be its balanced section-decomposition ($k\ge1$).
The \textit{binomial $B_{\bw}$ associated with $\bw$} is $B_{\bw}=B_{\bw}^+ -B_{\bw}^-$ where
\begin{eqnarray*}
B_{\bw}^+=\prod_{i:\text{even}}\prod_{e\in E({\bw}_{i})}e\qquad \text{and}\qquad B_{\bw}^-=\prod_{i:\text{odd}}\prod_{e\in E({\bw}_{i})}e.
\end{eqnarray*}
If $\bw$ has no unbalanced vertex term, then it is defined by $B_{\bw}^+=\prod_{e\in E(\bw)}e $ and $B_{\bw}^-=1$.
\end{definition}
Since $\bw$ has an even number of unbalanced vertex terms, its binomial is unique up to sign.
That is, the binomial is either $B_{\bw}$ or $-B_{\bw}$ according to its balanced section-decomposition.
For the even-signed closed walk ${\bw}$ in $(G,\tau)$ in Example~\ref{ex2}, $B_{\bw}=e_2e_4e_6-e_1e_3e_5e_6$ (one may say $B_{\bw}=e_1e_3e_5e_6-e_2e_4e_6$).

\begin{observation}\label{obervation}
If $\bw$ is an even-signed closed walk in a signed graph $(G,\tau)$, then $B_{\bw}\in I_{(G,\tau)}$.
\end{observation}

\begin{proof}
Let $\bw:v_{i_1}e_{j_1} \cdots v_{i_t} e_{j_t} v_{i_1}$ be an even-signed closed walk in $(G,\tau)$.
We may assume that the first vertex term is unbalanced, and let $\bw=\bw_0+\cdots+\bw_{2k-1}$ be a balanced section-decomposition of $\bw$.
For each edge term $e_{j_{\ell}}$, we let $\kappa(e_{j_{\ell}})=(-1)^{s}$ if $e_{j_{\ell}}$ belongs to the section $\bw_s$.

Let $\mathbb{b}=(b_e)_{e\in E(G)}$ be a vector such that $b_e=f^+(e)-f^{-}(e)$ for every edge $e$, where
\[f^+(e)= | \{\ell \mid e_{j_{\ell}}=e\text{ and }\kappa(e_{j_{\ell}})=1\} |,\qquad f^-(e)= | \{\ell \mid e_{j_{\ell}}=e\text{ and }\kappa(e_{j_{\ell}}) =-1 \}|.\]
Then the entry of $A\mathbb{b}$ corresponding to a vertex $v$ is
\[
\sum_{\substack{e: v\in e}} (f^+(e)-f^-(e)) \tau(e,v) =
\sum_{\substack{e: v\in e}} \sum_{\ell:e_{j_{\ell}}=e} \kappa(e_{j_{\ell}}) \tau(e,v)=\sum_{\ell:v_{i_{\ell}}=v} \left( \kappa(e_{j_{\ell-1}})\tau(e_{j_{\ell-1}},v) +\kappa(e_{j_{\ell}}) \tau(e_{j_{\ell}},v) \right).\]
If $\kappa(e_{j_{\ell}})=\kappa(e_{j_{\ell -1}})$, then $\tau(e_{j_{\ell-1}},v) =-\tau(e_{j_{\ell}},v)$, and if $\kappa(e_{j_{\ell}})\neq \kappa(e_{j_{\ell -1}})$, then  $\tau(e_{j_{\ell-1}},v) =\tau(e_{j_{\ell}},v)$.
In both cases, the sum $\kappa(e_{j_{\ell -1}})\tau(e_{j_{\ell -1}},v) +\kappa(e_{j_{\ell}})\tau(e_{j_{\ell}},v)$ is $0$. This implies that $\mathbb{e}^{\mathbb{b}^+}-\mathbb{e}^{\mathbb{b}^-}\in I_{(G,\tau)}$.

On the other hand, from the definition, $\mathbb{e}^{\mathbb{x}} (\mathbb{e}^{\mathbb{b}^+}-\mathbb{e}^{\mathbb{b}^-})=B_{\bw}$,
where $\mathbb{x}=(x_e)_{e\in E(G)}$ is the vector such that   $x_e=\min\{f^+(e),f^-(e)\}$ for every edge $e$.
Thus, $B_{\bw}$ is an element of $I_{(G,\tau)}$.
\end{proof}

It seems natural to have the following proposition from the definition of $I_{(G,\tau)}$.

\begin{proposition}
\label{prop:even_all_generator}
If $(G,\tau)$ is a signed graph, then the toric ideal $I_{(G,\tau)}$ is generated by
$$\{B_{\bw}\mid \bw\text{ is an even-signed closed walk in }(G,\tau)\}.$$
\end{proposition}

The above proposition immediately follows from  Observation~\ref{obervation} and Proposition~\ref{rmk:basic:binomial}, and we leave the proof detail  of Proposition~\ref{rmk:basic:binomial} in Appendix.

\begin{proposition}\label{rmk:basic:binomial}
Let $(G,\tau)$ be a signed graph,
and $\mathbb{b}={(b_e)}_{e\in E(G)}$ be a nonzero integer vector such that $A\mathbb{b}=\mathbb{0}$, where $A=A(G,\tau)$.
If we denote by $(G_{\mathbb{b}},\tau_{\mathbb{b}})$ the signed multigraph induced by $|b_e|$ copies of $e$ for every edge $e$ with its sign copied, then
each connected component $D$ of $G_{\mathbb{b}}$ has an  Eulerian $\bw_D$ which is an even-signed closed walk in $(G,\tau)$ and
\[ \mathbb{e}^{\mathbb{b}^+}-\mathbb{e}^{\mathbb{b}^-} =\prod_{\substack{ D:\text{connected}\\ \text{component of }G_{\mathbb{b}} }} B_{\bw_D}^+  \ -  \
\prod_{\substack{D:\text{connected} \\\text{component of }G_{\mathbb{b}}} } B_{\bw_D}^-.\]
\end{proposition}

\section{The main results}\label{sec:main}
In this section, we state the main results of the paper.
Subsection~\ref{subsec:generator} focuses on the  primitive binomials of $I_{(G,\tau)}$, and
Subsection~\ref{subsec:CI}  gives  characterizations of graphs $G$ with a complete intersection $I_{(G,\tau)}$ for every sign $\tau$.

\subsection{Primitive binomials of $I_{(G,\tau)}$}\label{subsec:generator}

We characterize all primitive binomials in $I_{(G,\tau)}$.

\begin{theorem}\label{thm:signed:primitive}
For an even-signed closed walk $\bw$ in a signed graph, $B_\bw$ is primitive if and only if the following hold:
\begin{itemize}
\item[\rm(i)]
The multigraph $[\bw]$ is constructed by $1$-clique sums of cycles of length at least two such that  every vertex of $[\bw]$ belongs to at most two cycles.
\item[\rm(ii)]
For every nontrivial section-decomposition $\bw=\bw_0+\bw_1$ into two closed walks $\bw_0$ and $\bw_1$,
each $\bw_i$ is odd-signed in $(G,\tau)$.
\end{itemize}
\end{theorem}

A proof of Theorem~\ref{thm:signed:primitive} is given in Subsection~\ref{sec:primitive}.
We often say an even-signed closed walk $\bw$ in a signed graph $(G,\tau)$ is \emph{primitive} if  $B_{\bw}$ is primitive in $I_{(G,\tau)}$.
Figure~\ref{fig:image} shows an image of the multigraph $[\bw]$ for a  primitive walk $\bw$. Note that if $\bw$ is a primitive walk,
then every cut vertex of $[\bw]$ decomposes $[\bw]$ into two parts, and each part corresponds to an odd-signed closed walk in $(G,\tau)$. \footnote{In the toric ideals of graphs, this was explained with a notion of `sink' of a block, see~\cite{RTT2012minimal}.}

\begin{figure}[ht!]
  \centering
  \includegraphics[width=5.5cm,page=4]{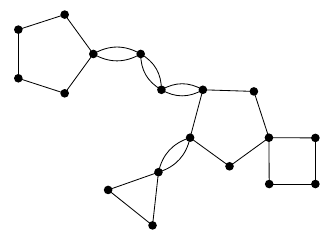}\\\caption{The multigraph $[\bw]$, where  $\bw$ is an even-signed closed walk in a signed graph $(G,\tau)$. If $(G,\tau)$ has no even-signed cycle, then $\bw$ is primitive. }\label{fig:image}
\end{figure}

\begin{example}\label{ex3}
Consider a signed graph $(G,\tau)$ in Figure~\ref{fig:sink}.
Let $\bw$ be a walk in $(G,\tau)$ defined by
$$\bw:v_1e_1v_2e_2v_3e_3v_1e_4v_4e_5v_5e_6v_6e_7v_7e_8v_5e_9 v_8 e_{10}v_1.$$
It has four balanced sections, and so $\bw$ is an even-signed closed walk in $(G,\tau)$.
See the vertex $v_5$ which is repeated in $\bw$. Then
$\bw$ has a nontrivial section-decomposition $\bw_0+\bw_1$ and each of $\bw_0$ and $\bw_1$ is an even-signed closed walk, where $\bw_0:~v_5e_6v_6e_7v_7e_8v_5$ and $\bw_1:~v_5e_9 v_8 e_{10}v_1e_1v_2e_2v_3e_3v_1e_4v_4e_5v_5$.
Thus, its associated binomial $B_{\bw}$ is not primitive by Theorem~\ref{thm:signed:primitive}(ii).
\begin{figure}[ht!]
  \centering
  \includegraphics[width=5cm,page=5]{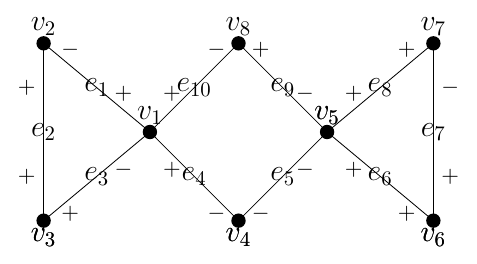}\\\caption{A signed graph $(G,\tau)$}\label{fig:sink}
\end{figure}
\end{example}

It is easy to see that Theorem~\ref{thm:signed:primitive} is a generalization of Theorem~\ref{thm:graph:primitive}.
Moreover, if you consider an orientation of a cycle as a signed graph, then it is an even-signed closed walk. Thus, every digraph has no odd-signed closed walk,
which implies that
Theorem~\ref{thm:signed:primitive} is  also a generalization of Theorem~\ref{thm:digraph:primitive}.

\subsection{The complete intersection property of the toric ideal $I_{(G,\tau)}$ }\label{subsec:CI}

We compute the rank of the incidence matrix of a connected signed graph first.

\begin{proposition}\label{prop:rank}
Let $(G,\tau)$ be a connected signed graph.
Then \[\mathrm{rank}(A(G,\tau))=\begin{cases}
|V(G)|-1 &\text{if there is no odd-signed closed walk in }(G,\tau),\\
|V(G)| &\text{otherwise.}
\end{cases}\]
\end{proposition}

See Appendix for the proof of Proposition~\ref{prop:rank}.  Due to the proposition and from the fact that ht$(I_{(G,\tau)})=m-\mathrm{rank}(A{(G,\tau)})$, we define the following.

\begin{definition}
Let $(G,\tau)$ be a connected signed graph. We say $(G,\tau)$ is a \textit{complete intersection} when $I_{(G,\tau)}$ is a (binomial) complete intersection, i.e., $I_{(G,\tau)}$ is generated by $r({G,\tau})$ binomials, where
\[r(G,\tau)=\begin{cases}
  |E(G)|-|V(G)|+1  &\text{if there is no odd-signed closed walk in }(G,\tau),\\
  |E(G)|-|V(G)| &\text{otherwise.}
\end{cases}\]
\end{definition}

For a disconnected signed graph, it is said to be a \textit{complete intersection} if every connected component is a complete intersection.
Let $\mathcal{G}^{cis}$ be the set of all graphs $G$ such that
$(G,\tau) $ is a (binomial) complete intersection for every sign $\tau$.
It is natural to ask which graphs are in $\mathcal{G}^{cis}$, and we start from basic observations.

\begin{proposition}\label{lem:each:block}
Every block  of a graph in $\mathcal{G}^{cis}$ belongs to $\mathcal{G}^{cis}$.
\end{proposition}

\begin{proof}
Let $G$ be a   graph in $\mathcal{G}^{cis}$ having a block $H$ such that $H\not\in \mathcal{G}^{cis}$. We may assume that $G$ is connected.
Then there is a sign $\tau$ of $H$ such that $I_{(H,\tau)}$ cannot be generated by $r(H,\tau)$ binomials. Let $t$ be the minimum number of binomials which generate $I_{(H,\tau)}$.
Then $t>r(H,\tau)$.

Let $\tau'$ be the sign of $G$ such that $\tau'(e,v)=\tau(e,v)$ for every $e\in E(H)$
and $\tau'(e,v)\tau'(e,w)=-1$ for every $e=vw\in E(G)\setminus E(H)$.
Let $X=\{ v\in V(G)\mid v\text{ is contained in a block other than }H\}$.
Since $G\in \mathcal{G}^{cis}\subset \mathcal{G}^{cio}$, it follows that $G[X]\in \mathcal{G}^{cio}$ by Theorem~\ref{thm:CI:digraph}(ii).
Note that $(G[X],\tau'|_X)$  can be understood as a digraph, and so $(G[X],\tau'|_X)$ has no odd-signed cycle.
Thus, \begin{eqnarray}\label{eq:r:block}
&&r(G,\tau')=r(H,\tau)+(|E(G[X])|-|X|+c)=r(H,\tau)+(|E(G)|-|E(H)|-|X|+c),\end{eqnarray} where $c=|V(H)\cap X|$.
Moreover, since $G[X]$ has at least $c$ components, we need
at least $|E(G)|-|E(H)|-|X|+c$ binomials to generate
$I_{(G[X],\tau'|_X)}$.
Hence, in order to generate $I_{(G,\tau')}$, we need at least $t+|E(G)|-|E(H)|-|X|+c$ binomials.
Since $t+|E(G)|-|E(H)|-|X|+c> r(G,\tau')$ by \eqref{eq:r:block},
we reach a contradiction to the fact that $G\in \mathcal{G}^{cis}$.
\end{proof}

\begin{proposition}\label{lem:sumevencycle}
Let $H$ be either a cycle or a $K_2$. For a connected graph $G$, let $G'$ be a $1$- or $2$-clique sum of $G$ and $H$. Let
$\tau'$ be a sign of $G'$ such that $\tau=\tau'|_{V(G)}$.
Suppose that $H$ is even-signed in $(G',\tau')$ when $H$ is a cycle.
Then $I_{(G',\tau')}$ is a complete intersection if and only if $I_{(G,\tau)}$ is a complete intersection.
\end{proposition}

\begin{proof} Suppose that $H$ is a cycle.
It is clear that
$r(G',\tau')=r(G,\tau)+1$.
Since $C$ is a primitive walk in $(G',\tau')$, its associated binomial generates $I_{(G',\tau')}$ together with a generating set of $I_{(G,\tau)}$.
If $H=K_2$, then $r(G',\tau')=r(G,\tau)$ for every sign $\tau'$ of $G'$ and the primitive walks of $(G,\tau)$ and those of $(G',\tau')$ are the same. Thus, the proposition holds.
\end{proof}

Proposition \ref{lem:sumevencycle}  implies that
for a graph $G$, if $G\not\in \mathcal{G}^{cis}$, then a graph constructed by clique sums of $G$ and cycles/$K_2$ is not in $\mathcal{G}^{cis}$. Thus the following holds.

\begin{corollary}\label{cor:sumevencycle}
For a graph $G\in \mathcal{G}^{cis}$ and an induced subgraph $H$ of $G$, if $G$ can be constructed by clique sums of $H$ and cycles/$K_2$, then $H$ belongs to $\mathcal{G}^{cis}$.
\end{corollary}

\begin{observation}\label{lem:K4}
A graph in $\mathcal{G}^{cis}$ is $K_4$-free.
\end{observation}

\begin{proof}
First, we show that $K_4\not\in \mathcal{G}^{cis}$. Following the labeling in Figure~\ref{fig:K4}, let $\tau$ be a sign of $K_4$ such that
\[\tau(e_4,v_1)=\tau(e_5,v_2)=\tau(e_6,v_3)=-1,\]
and all the others have sign $1$.
By Theorem~\ref{thm:signed:primitive}, there are only three primitive walks $\bw_1$, $\bw_2$, $\bw_3$, which are defined as Figure~\ref{fig:K4}.\footnote{
We note that Observation~\ref{lem:K4} is not used in the proof of Theorem~\ref{thm:signed:primitive}.
Moreover, in Figure~\ref{fig:K4} (also in the following figures of the paper), we use dashed lines and gray color to draw the rest part of the graph not belonging to  $[\bw_i]$ together to distinguish the walks easily.}
Then we have
$B_{\bw_1}= e_1e_5-e_3e_6$,  $B_{\bw_2}= e_1e_4-e_2e_6$, and $B_{\bw_3}= e_2e_5-e_3e_4$.
Note that each of three cannot be generated by the others. However,
$r(K_4,\tau)=2$, since $(K_4,\tau)$ has odd-signed cycles. Thus, $I_{(K_4,\tau)}$ is not a complete intersection.
\begin{figure}[h!]
  \centering
 \includegraphics[width=12cm,page=6]{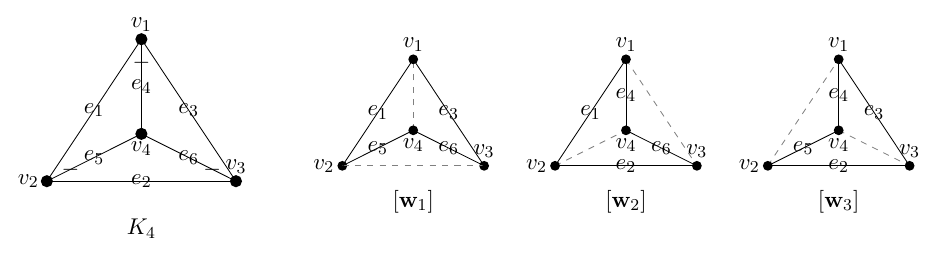}\\
  \caption{A complete graph $K_4$ and its three closed walks}
  \label{fig:K4}
\end{figure}

Suppose that there is a graph $G$  in $\mathcal{G}^{cis}$ having $K_4$ as a subgraph. We take such $G$ as a smallest one.
By Proposition~\ref{lem:each:block}, a block of $G$ with $K_4$ is in  $\mathcal{G}^{cis}$.
Thus $G$ is $2$-connected. Moreover, by the above argument, $G\neq K_4$.
Since $G\in \mathcal{G}^{cio}$, $G$ is constructed by clique sums of cycles and/or complete graphs by Theorem~\ref{thm:CI:digraph}(i).
Since $G$ is 2-connected, it follows that every clique sum to  construct  $G$ is not a 1-clique sum.
From the fact that $G\in \mathcal{G}^{ci}$,
by Theorem~\ref{graph:useful:facts}(ii),
it follows that every clique sum to construct $G$ is a 2-clique sum.
If $G$ is constructed by clique sums of exactly one $K_4$ and cycles, then by Proposition~\ref{lem:sumevencycle} and the fact that $K_4\not\in \mathcal{G}^{cis}$, it follows that $G\not\in \mathcal{G}^{cis}$, a contradiction.
Thus $G$ has at least two $K_4$, say $K$ and $K'$.
By Theorem~\ref{graph:useful:facts}(v),
$G[K\cup K']$ is in $\mathcal{G}^{ci}$.
By applying Theorem~\ref{graph:useful:facts}(i) to $G[K\cup K']$, we know that $G[K\cup K']$ is a disjoint union of $K$ and $K'$.
By Theorem~\ref{graph:useful:facts}(iv), every two vertex disjoint cycles of length three from $K$ and $K'$ are connected by two disjoint edges, which is a contradiction.
\end{proof}

We remark that $\mathcal{G}^{cis}$ is a subset of $\mathcal{G}^{ci}\cap \mathcal{G}^{cio}$ by definitions, and so
Observation~\ref{lem:K4} tells us from Theorem~\ref{thm:CI:digraph}(i) that the (connected) graphs in  $\mathcal{G}^{cis}$ are constructed by clique sums of cycles and/or $K_2$. Among those graphs, we completely characterize all graphs  in $\mathcal{G}^{cis}$.
The following considers only $2$-connected graphs in $\mathcal{G}^{cis}$, and its proof is given in
Subsection~\ref{sec:ci}.

\begin{theorem}\label{thm:characterization:CI}
For a $2$-connected graph $G$ with at least three vertices, $G\in\mathcal{G}^{cis}$ if and only if $G$ is one of {\rm(G1)}-{\rm(G5)} for some $m,n\ge 3$ (see Figure~\ref{fig:CIS_list}):
\begin{itemize}
\item[\rm (G1)] A cycle $C_n$;
\item[\rm (G2)] A $2$-clique sum of two cycles $C_n$ and $C_m$;
\item[\rm (G3)] ${}_m\CI_{n}$: the graph obtained from $C_3$ by gluing $C_m$ and $C_n$ to two distinct edges of $C_3$ using $2$-clique sum, respectively;
\item[\rm (G4)] ${}_m\CII_n$: the graph obtained from $C_4$ by gluing $C_m$ and $C_n$ to two opposite edges of $C_4$ using $2$-clique sum, respectively;
\item[\rm (G5)] ${}_m\CIII_n$: the graph obtained from ${}_m\CII_n$ by adding a diagonal edge of the middle $C_4$.
\end{itemize}
\end{theorem}

\begin{figure}[h!]
\centering
\includegraphics[width=10cm,page=7]{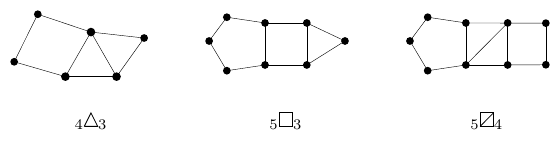}\\
\caption{Some graphs in Theorem~\ref{thm:characterization:CI}}
 \label{fig:CIS_list}
\end{figure}

Now we characterize all graphs in $\mathcal{G}^{cis}$.
The  proof  of  Theorem~\ref{thm:cis:general:graph}  is  given  in  Subsection~\ref{sec:general}, and see Figure~\ref{fig:cis:all} for some graphs described in the theorem.

\begin{theorem}\label{thm:cis:general:graph}
For a graph $G$, $G$ is in $\mathcal{G}^{cis}$ if and only if every connected component $G'$ of $G$ is one of the following:
\begin{itemize}
\item[\rm (i)] $G'$ is a tree.
\item[\rm (ii)] $G'$ has exactly one nonedge block and it is isomorphic to one of {\rm(G1)}$\sim${\rm(G5)}.
\item[\rm (iii)] $G'$ has exactly two nonedge blocks $B$ and $B'$, each of which is isomorphic to {\rm(G1)} or {\rm(G2)}.
When $B$ is {\rm(G2)},
the vertex $v$ of $B$ closest to $B'$ is on a triangle of $B$ and $\deg_{B}(v)=2$.
\end{itemize}
\end{theorem}
\begin{figure}[ht!]
  \centering
 \includegraphics[width=17cm,page=8]{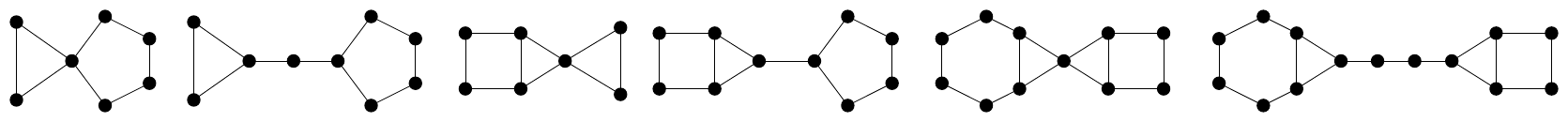}\\
  \caption{Examples of graphs in $\mathcal{G}^{cis}$ satisfying Theorem~\ref{thm:cis:general:graph}(iii)} \label{fig:cis:all}
\end{figure}

We remark that from the structures of the graphs in Theorem~\ref{thm:cis:general:graph}, it follows that every induced subgraph of a graph in $\mathcal{G}^{cis}$ belongs to $\mathcal{G}^{cis}$.

Now, we finish the section by noting that it is not difficult to find graphs in $\mathcal{G}^{ci}\cap \mathcal{G}^{cio}$ which are not in $\mathcal{G}^{cis}$. A reader may already notice that $K_4$ is such an example by Theorems~\ref{thm:K4:ci}~and~\ref{thm:CI:digraph}(i), and Observation~\ref{lem:K4}.
The following properties not only are helpful to understand Example~\ref{ex} but also may give an idea to find graphs in $(\mathcal{G}^{ci}\cap \mathcal{G}^{cio})\setminus \mathcal{G}^{cis}$.

Let $\mathbf{p}:v_{0}e_{1}v_{1}\cdots e_{t}v_{t}$
be a path in $G$ of length at least two such that $v_{0}v_{t}\not\in E(G)$ and $\deg_G(v_i)=2$ for each $i \in[{t-1}]$ as depicted in Figure \ref{fig:ear}.
We call such path an \textit{ear} of $G$, and
we denote by $G/\mathbf{p}$ the graph obtained from $G$ by deleting the vertices $v_{1}$, $\ldots$, $v_{t-1}$ and adding an edge between $v_{0}$ and $v_{t}$. We sometimes call $G/\mathbf{p}$ a \textit{contraction} of $G$ by $\mathbf{p}$. The proof of the following proposition is given in Appendix.

\begin{proposition}\label{lem:that_lemma}
Let  $\mathbf{p}:v_{0}e_{1}v_{1}\cdots e_{t}v_{t}$ ($t\ge 2$) be an ear of a graph  $G$.
Then the following hold:
\begin{itemize}
    \item[\rm(i)] If $G \in \mathcal{G}^{cis}$, then $G/\mathbf{p} \in \mathcal{G}^{cis}$.
    (Equivalently, if a graph is not in $\mathcal{G}^{cis}$, then its subdivision is not in $\mathcal{G}^{cis}$.)
    \item[\rm(ii)] If $t\ge 3$ and $G/\mathbf{q} \in \mathcal{G}^{cis}$, where $\mathbf{q}=\mathbf{p}-v_t$, then $G \in \mathcal{G}^{cis}$.  (Equivalently, if  $G\not\in\mathcal{G}^{cis}$, then the graph obtained by contracting an edge $e=uv$ with $\deg_G(u)=\deg_G(v)=2$ is not in $\mathcal{G}^{cis}$.)
    \end{itemize}
\end{proposition}

\begin{example}\label{ex}
Let $G$ be the graph in Figure~\ref{fig:EXAMPLE2}. Then $G\not\in \mathcal{G}^{cis}$ and
$G\in \mathcal{G}^{ci}\cap \mathcal{G}^{cio}$.

First, we consider the graph $G_0$ in Figure~\ref{fig:EXAMPLE2}.
Note that if $G_0\not\in \mathcal{G}^{ci}$ then $G_0\not\in \mathcal{G}^{cis}$, which also implies that by  Proposition~\ref{lem:that_lemma}(i), $G\not\in \mathcal{G}^{cis}$.
Thus, it is sufficient to show that $G_0\not\in \mathcal{G}^{ci}$ by Algorithm~\ref{thm:CI_Graphs}.
Note that for each $i\in\{1,3,5\}$,
vertex $v_i$ of $G_0$ has degree two and $b(G_0-v_i)=b(G_0)$.
We apply the algorithm to $G_0$ with $v_1$.
Then $W=\{v_1,v_2,v_4,v_6\}$ and a shortest closed walk $\bw_1$ of even length with $V(\bw_1)=W$ is a cycle of length 4. Its associated binomial is $B_{\bw_1}=e_1e_8-e_6e_{7}$.
Similarly, by considering the vertex $v_3$  and $v_5$ one by one, finally, we have $\mathcal{B}=\{B_{\bw_1},B_{\bw_2},B_{\bw_3}\}$ where  $B_{\bw_2}=e_2e_8-e_3e_9$ and
$B_{\bw_3}=e_4e_9-e_5e_7$. It remains to check if $I_G=\left< \mathcal{B} \right>$.
However, $B_{\bw_4}$ cannot be generated by $\mathcal{B}$, where $B_{\bw_4}=e_1e_3e_5-e_2e_4e_6$ is the primitive binomial associated with $\bw_4: v_1e_1v_2e_2v_3e_3v_4e_4v_5e_5v_6e_6v_1$.
Thus, Algorithm~\ref{thm:CI_Graphs} returns \textsc{False}, as a desired one.

\begin{figure}[ht!]
  \centering
 \includegraphics[width=17cm,page=9]{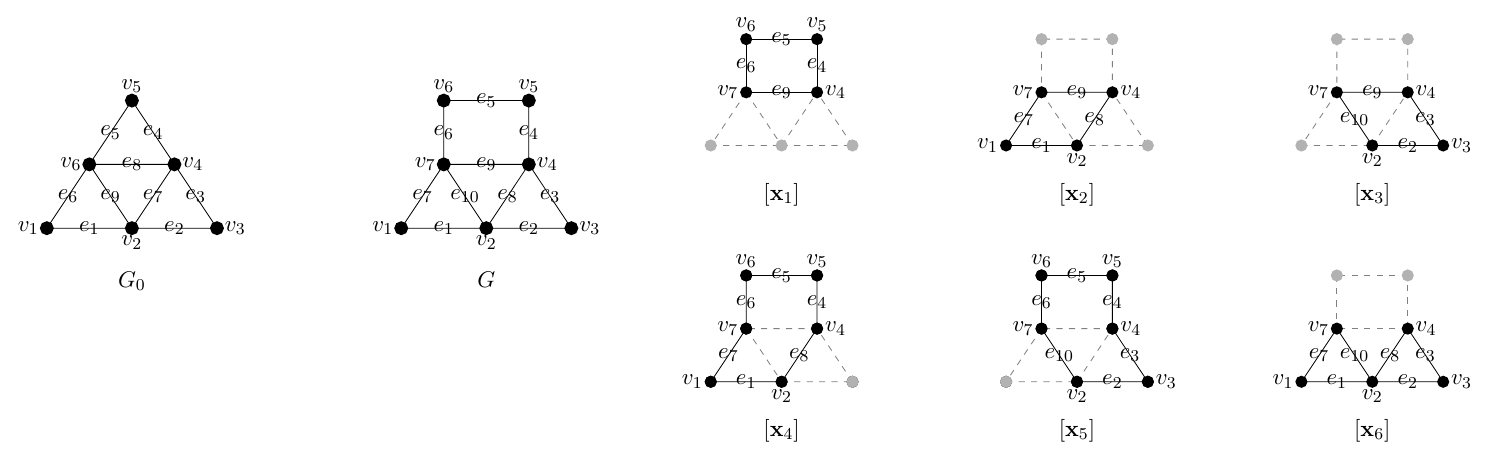}\\
 \caption{Graphs $G_0$, $G$ and primitive walks in $G$}\label{fig:EXAMPLE2}
\end{figure}

Now, we will show that
$G\in \mathcal{G}^{ci}\cap \mathcal{G}^{cio}$. Since $G$ is constructed by clique sums of cycles, $G\in\mathcal{G}^{cio}$ by Theorem~\ref{thm:CI:digraph}(i). It remains to check that $G\in \mathcal{G}^{ci}$.
Note that $G$ has exactly six primitive walks $\bx_1\sim \bx_6$, defined as  Figure~\ref{fig:EXAMPLE2}.
Then one can check  from Corollary~\ref{rmk2:Lem4_1} that
\[B_{\bx_4}\in \left<  B_{\bx_1},B_{\bx_2}\right>, \qquad
B_{\bx_5}\in \left<B_{\bx_1}, B_{\bx_3}\right>, \qquad
B_{\bx_6}\in \left<B_{\bx_2}, B_{\bx_3}\right>.\]
Thus, $I_G=\left<B_{\bx_1},B_{\bx_2},B_{\bx_3}\right>$, which implies that $I_G$ is a complete intersection.
\end{example}

\section{Properties of walks in a signed graph}\label{sec:property}

In this section, we investigate properties of even-signed closed walks in a signed graph $(G,\tau)$, which play an important role in the following section.

\begin{lemma}\label{eq:sign:product:closed}
In a signed graph $(G,\tau)$, for two closed walks  $\bw$ and $\bw'$ sharing a vertex,
$\mu(\bw+\bw')=\mu(\bw)\mu(\bw')$.
\end{lemma}
\begin{proof}
If $\bw$ or $\bw'$ is trivial, then it is clear.
Suppose that both are nontrivial.
Let $\bw: ue_{j_1}v_{i_2}\cdots e_{j_t} u$ and $\bw': ue'_{j_1}v'_{i_2}\cdots e'_{j_r} u$.
Since $\tau(e_{j_1},u)\tau(e_{j_t},u)=1$ (resp. $\tau(e'_{j_1},u)\tau(e'_{j_r},u)=1$) means that $u$ is an unbalanced vertex term of $\bw$ (resp. $\bw'$), $\mu(\bw+\bw')$ is equal to
\[(-\tau(e_{j_1},u)\tau(e_{j_t},u))(-\tau(e'_{j_1},u)\tau(e'_{j_r},u))\mu(\bw)\mu(\bw')
(-\tau(e_{j_t},u)\tau(e'_{j_1},u))(-\tau(e'_{j_r},u)\tau(e_{j_1},u))=\mu(\bw)\mu(\bw').\]
\end{proof}

\begin{lemma}\label{lem:twowalks:basic}
Let $\bw$ be a $(u,v)$-walk in a signed graph   $(G,\tau)$.
For any two $(v,u)$-walks $\bw_1$ and $\bw_2$,
\[\mu(\bw_1+\bw_2^{-1})=\mu(\bw+\bw_1)\mu(\bw+\bw_2).\]
\end{lemma}
\begin{proof}
If $u=v$, then the lemma holds, since we have the following from Lemma~\ref{eq:sign:product:closed}:
\begin{eqnarray*}
&& \mu(\bw_1+\bw_2^{-1})=\mu(\bw_1)\mu(\bw_2^{-1})=\mu(\bw_1)\mu(\bw_2)= \mu(\bw)\mu(\bw_1)\mu(\bw)\mu(\bw_2)= \mu(\bw+\bw_1)\mu(\bw+\bw_2).
\end{eqnarray*}
Suppose that $u$ and $v$ are distinct. Then each of $\bw$, $\bw_1$, and $\bw_2$ is nontrivial. Without loss of generality, it is enough to consider the case when
\[\bw: ue_1v_2\cdots e_tv,\qquad\bw_1: ve'_1v'_2\cdots e'_ru,\qquad \bw_2: ve''_1v''_2\cdots e''_su\]
as depicted in Figure~\ref{fig:Lemma41}.

\begin{figure}[h!]
  \centering  \includegraphics[width=5.5cm,page=10]{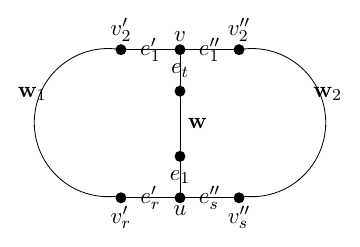}\\\caption{An illustration for the proof of Lemma~\ref{lem:twowalks:basic}}
   \label{fig:Lemma41}
   \end{figure}

Note that $\tau(e'_r,u)\tau(e''_s,u)=1$ (resp. $\tau(e'_1v)\tau(e''_1,v)=1$) means that $u$ (resp. $v$) is a new unbalanced vertex term of $\bw_1+\bw_2^{-1}$.
Thus,
\[\mu(\bw_1+\bw_2^{-1})=\mu(\bw_1)(-\tau(e'_r,u)\tau(e''_s,u))\mu(\bw_2)(-\tau(e'_1,v)\tau(e''_1,v)).\]
Likewise,
\begin{eqnarray*}
\mu(\bw+\bw_1)&=&\mu(\bw)(-\tau(e_t,v)\tau(e'_1,v))\mu(\bw_1)(-\tau(e'_r,u)\tau(e_1,u)),\\
\mu(\bw+\bw_2)&=&\mu(\bw)(-\tau(e_t,v)\tau(e''_1,v))\mu(\bw_2)(-\tau(e''_s,u)\tau(e_1,u)).
\end{eqnarray*}
Therefore,
$\mu(\bw_1+\bw_2^{-1})=\mu(\bw_1)\mu(\bw_2)\tau(e'_r,u)\tau(e''_s,u)
\tau(e'_1,v)\tau(e''_1,v)=\mu(\bw+\bw_1)\mu(\bw+\bw_2)$.
\end{proof}

\begin{lemma}\label{lem:product}
Let $\bw$ and $\bw'$ be two even-signed closed walks in a signed graph, whose first vertex terms are the same.
Then (by taking $B_{\bw}$ and $B_{\bw'}$ properly)
$B^+_{\bw+\bw'}=B^+_{\bw}B^+_{\bw'}$ and $B^-_{\bw+\bw'}=B^-_{\bw}B^-_{\bw'}$.
\end{lemma}

\begin{proof}
Let $\bw=\bw_0+\cdots+\bw_{r}$ and $\bw'=\bw'_0+\cdots+\bw'_s$ be balanced section-decompositions of $\bw$ and $\bw'$ for some odd integers $r$ and $s$.
By Lemma~\ref{eq:sign:product:closed},  $\bw_{r}+\bw'_0$ is a balanced section of $\bw+\bw'$ if and only if $\bw'_{s} +\bw_0$ is a balanced section of $\bw+\bw'$.
First, suppose that $\bw_{r}+\bw'_0$ is not a balanced section of $\bw+\bw'$.
Then $\bw_0+\cdots+\bw_{r}+\bw'_0+\cdots+\bw'_s$ is a balanced section-decomposition of $\bw+\bw'$. By definition,
it follows that $B^+_{\bw}B^+_{\bw'}=B^+_{\bw+\bw'}$ and
$B^-_{\bw}B^-_{\bw'}=B^-_{\bw+\bw'}$.
Suppose that $\bw_{r}+\bw'_0$ is a balanced section of $\bw+\bw'$.
Letting $\mathbf{x}=\bw'_s+\bw_0$ and $\mathbf{y}=\bw_r+\bw_0'$,
we have a balanced section-decomposition of $\bw+\bw'$,
$$\mathbf{x}+\bw_1+\cdots+\bw_{r-1}+\mathbf{y}+\bw'_1+\cdots+\bw'_{s-1}.$$
Then we can obtain that
$B^+_{\bw}B^+_{\bw'}=B^+_{\bw+\bw'}$ and $B^-_{\bw}B^-_{\bw'}=B^-_{\bw+\bw'}$
by redefining $B_{\bw}$ and $B_{\bw'}$ properly.
\end{proof}

\begin{lemma}\label{lem:moreover:part}
Let $\bw$ be a $(u,v)$-walk in a signed graph without unbalanced vertex term.
For two $(v,u)$-walks $\bw_1$ and $\bw_2$,
if $\bw+\bw_i$ is even-signed for each $i=1,2$,
then $B_{\bw_1+\bw_2^{-1}}$ belongs to the ideal $\left< B_{\bw+\bw_1}, B_{\bw+\bw_2} \right>$.
\end{lemma}

\begin{proof}
For simplicity, let $\mathbf{a}$, $\mathbf{b}$, $\mathbf{c}$, and $\mathbf{d}$ be closed walks such that
\[\begin{array}{llll}
\mathbf{a}=\bw+\bw_1, \quad&
\mathbf{b}=\bw_2^{-1}+\bw^{-1},\quad& \mathbf{c}=\bw_1+\bw_2^{-1},\quad
&\mathbf{d}=\bw^{-1}+\bw.
\end{array}\]
Note that each of the four walks are even-signed ($\mathbf{a}$ and $\mathbf{b}$ are even-signed by the assumptions, $\mathbf{c}$ and $\mathbf{d}$ are even-signed by Lemma~\ref{lem:twowalks:basic}).
Moreover, the first vertex terms of $\mathbf{a}$ and $\mathbf{b}$ are the same as $u$, and the first vertex terms of $\mathbf{c}$ and $\mathbf{d}$ are the same as $v$.

We consider two closed walks $\mathbf{a}+\mathbf{b}$ and $\mathbf{c}+\mathbf{d}$.
Since
$\mathbf{a}+\mathbf{b}=\bw+\bw_1+(\bw_2)^{-1}+\bw^{-1}$ and $\mathbf{c}+\mathbf{d}=\bw_1+(\bw_2)^{-1}+\bw^{-1}+\bw$,  they are the same walk and so
$\pm B_{\mathbf{a}+\mathbf{b}}= B_{\mathbf{c}+\mathbf{d}}$.
By Lemma~\ref{lem:product}, by taking binomials associated with four even-signed closed walks $\mathbf{a}$, $\mathbf{b}$,
$\mathbf{c}$ and $\mathbf{d}$  properly, we have
\[B^+_{\mathbf{a}+\mathbf{b}}=B^+_{\mathbf{a}}B^+_{\mathbf{b}}, \qquad B^-_{\mathbf{a}+\mathbf{b}}=B^-_{\mathbf{a}}B^-_{\mathbf{b}}, \qquad
B^+_{\mathbf{c}+\mathbf{d}}=B^+_{\mathbf{c}}B^+_{\mathbf{d}}, \qquad B^-_{\mathbf{c}+\mathbf{d}}=B^-_{\mathbf{c}}B^-_{\mathbf{d}}.\]
Note that since $\bw$ has no unbalanced vertex term, $B^+_{\mathbf{d}}=B^-_{\mathbf{d}}$, and we let $X:=B^+_{\mathbf{d}}$. Hence,
{\small \begin{eqnarray*}
\!\!\!\!\!\!B_{\mathbf{c}}\!\!\!&=&\!\!\!B^+_{\mathbf{c}}-B^-_{\mathbf{c}} = \frac{B^+_{\mathbf{c}+\mathbf{d}}}{B^+_{\mathbf{d}}}
-\frac{B^-_{\mathbf{c}+\mathbf{d}}}{B^-_{\mathbf{d}}}
= \frac{1}{X} \left(B^+_{\mathbf{c}+\mathbf{d}}-B^-_{\mathbf{c}+\mathbf{d}}\right)=\frac{1}{X} B_{\mathbf{c}+\mathbf{d}}=\pm\frac{1}{X}B_{\mathbf{a}+\mathbf{b}}=\pm
\frac{B^+_{\mathbf{a}}B^+_{\mathbf{b}}}{X} \mp \frac{B^-_{\mathbf{a}}B^-_{\mathbf{b}}}{X}. \end{eqnarray*}}
We may assume that
$B_{\mathbf{c}}=\frac{B^+_{\mathbf{a}}B^+_{\mathbf{b}}}{X}- \frac{B^-_{\mathbf{a}}B^-_{\mathbf{b}}}{X}.$ (The other case is similar.)
Thus, $X$ divides both $B^+_{\mathbf{a}}B^+_{\mathbf{b}}$ and $B^-_{\mathbf{a}}B^-_{\mathbf{b}}$.
Moreover, since $\bw$ has no unbalanced vertex term,
$X$ divides one of $B^+_{\mathbf{a}}$ and $B^-_{\mathbf{a}}$, and one of
 $B^+_{\mathbf{b}}$ and $B^-_{\mathbf{b}}$.
Thus, $X$ divides either  $B^+_{\mathbf{a}}$ and $B^-_{\mathbf{b}}$, or
 $B^-_{\mathbf{a}}$ and $B^+_{\mathbf{b}}$.
If $X$ divides $B^+_{\mathbf{a}}$ and $B^-_{\mathbf{b}}$, then
{\small \begin{eqnarray*}
B_{\mathbf{c}}&=&
\frac{B^+_{\mathbf{a}}B^+_{\mathbf{b}}}{X}+\left(
-\frac{B^+_{\mathbf{a}}B^-_{\mathbf{b}}}{X}+\frac{B^+_{\mathbf{a}}B^-_{\mathbf{b}}}{X}
\right)-\frac{B^-_{\mathbf{a}}B^-_{\mathbf{b}}}{X}\\
&=&
\left( \frac{B^+_{\mathbf{a}}B^+_{\mathbf{b}}}{X}
-\frac{B^+_{\mathbf{a}}B^-_{\mathbf{b}}}{X} \right)+ \left( \frac{B^+_{\mathbf{a}}B^-_{\mathbf{b}}}{X}
-\frac{B^-_{\mathbf{a}}B^-_{\mathbf{b}}}{X}\right)=\frac{B^+_{\mathbf{a}}}{X} B_{\mathbf{b}} +
\frac{B^-_{\mathbf{b}}}{X} B_{\mathbf{a}}.
\end{eqnarray*}}
Similarly, if $X$ divides $B^-_{\mathbf{a}}$ and $B^+_{\mathbf{b}}$, then
{\small \begin{eqnarray*}
B_{\mathbf{c}}&=&
\frac{B^+_{\mathbf{a}}B^+_{\mathbf{b}}}{X}+\left(
-\frac{B^-_{\mathbf{a}}B^+_{\mathbf{b}}}{X}+\frac{B^-_{\mathbf{a}}B^+_{\mathbf{b}}}{X}
\right)-\frac{B^-_{\mathbf{a}}B^-_{\mathbf{b}}}{X}
=\frac{B^-_{\mathbf{a}}}{X} B_{\mathbf{b}} + \frac{B^+_{\mathbf{b}}}{X} B_{\mathbf{a}}.
\end{eqnarray*}}
In any case, $B_{\mathbf{c}}$ belongs to the ideal $\left< B_{\mathbf{a}}, B_{\mathbf{b}} \right>$, a desired conclusion.
\end{proof}

The following is from  Lemma~\ref{lem:moreover:part} by considering cases where $\bw$ is a walk of length one.

\begin{corollary}\label{rmk2:Lem4_1}
Let $\bw_1$ and $\bw_2$ be two even-signed closed walks in a signed graph, starting with $u,e,v$ for an edge $e=uv$. Then $B_{\bw'}\in \left< B_{\bw_1},B_{\bw_2}\right>$,
where $\bw_i=uev+\bw'_i$ for $i=1,2$ and $\bw'= \bw'_1+{\bw'_2}^{-1}$.
\end{corollary}

The following lemma may fail if we drop the assumption on oddness of sign of $\bw$ or $\bw'$.

\begin{lemma}\label{lem:twowalks:reverse}
Let $\bw$ and $\bw'$ be two odd-signed closed walks in a signed graph  $(G,\tau)$, whose first vertex terms are the same.
Then $\bw+\bw'$ and $\bw^{-1}+\bw'$ are even-signed closed walks in $(G,\tau)$ and  $B_{\bw^{-1}+\bw'}=\pm B_{\bw+\bw'}$.
\end{lemma}

\begin{proof}
Note that  $\bw+\bw'$  and $\bw^{-1}+\bw'$ are even-signed by
Lemma~\ref{eq:sign:product:closed}, since $\mu(\bw^{-1})=\mu(\bw)=\mu(\bw')=-1$,
$\mu(\bw+\bw')=\mu(\bw)\mu(\bw')=1$, and $\mu(\bw^{-1}+\bw')=\mu(\bw^{-1})\mu(\bw')=1$.
Let $\bw_0+\cdots+\bw_{2r}$ be a balanced section-decomposition of $\bw$ for some nonnegative integer $r$.
We assume that $\bw_0$ contains the first vertex term $v$. Then $\mathbf{y}$ is a nontrivial walk and $\mathbf{x}$ may be trivial, and
$\bw+\bw'$ and $\bw^{-1}+\bw'$ have the following section-decompositions (the parts $\bw_1+\cdots+\bw_{2r}$ and $\bw_{2r}^{-1}+\cdots+\bw_1$ are dropped if $r=0$):
\[\begin{array}{ccccccccccccccc}
\bw&+&\bw'&=&\mathbf{y}&+&\bw_1&+&\cdots&+&\bw_{2r}&+&\mathbf{x}&+&\bw',
\\
\bw^{-1}&+&\bw'&=&\mathbf{x}^{-1}&+&\bw^{-1}_{2r}&+&\cdots&+&\bw^{-1}_{1}&+&\mathbf{y}^{-1}&+&\bw'.
\end{array}\]
Then we make a binomial $B_{\bw+\bw'}=B^+-B^-$ by putting the edges in $E(\mathbf{y})$ to $B^+$,  the edges in $E(\bw_1)$ to $B^-$,  the edges in $E(\bw_2)$ to $B^+$, and so on.
Let $U^+$ and $U^-$ be the (multi)set so that
\[ B^+= \prod_{e\in U^+}e \qquad \text{and}\qquad B^-=\prod_{e\in U^-}e.\]
Then  by a way to make the binomial $B_{\bw+\bw'}$,
\[ U^+\supset  E(\mathbf{y}) \cup\left(  \bigcup_{i>0:\text{ even}} E(\mathbf{w}_i) \right),\qquad
U^- \supset  \left( \bigcup_{i:\text{ odd}} E(\mathbf{w}_i) \right)
\cup E(\mathbf{x}). \]
Similarly, we also make a binomial $B_{\bw^{-1}+\bw'}=B'^+-B'^-$ by putting the edges in $E(\mathbf{x}^{-1})$ to $B'^-$, the edges in  $E(\bw_{2r}^{-1})$ to $B^+$, and so on.
Let $W^+$ and $W^-$ be the (multi)set so that
\[ B'^+= \prod_{e\in W^+}e \qquad \text{and}\qquad B'^-=\prod_{e\in W^-}e.\]
Then by a way to make the binomial $B_{\bw^{-1}+\bw'}$,
\[ W^+\supset  E(\mathbf{y}) \cup\left(  \bigcup_{i>0:\text{ even}} E(\mathbf{w}_i) \right),\qquad
W^- \supset  \left( \bigcup_{i:\text{ odd}} E(\mathbf{w}_i) \right)
\cup E(\mathbf{x}). \]
Note that it is sufficient to show that for the first edge term $e'_{j_1}$ of $\bw'$,
$e'_{j_1}\in U^+$ if and only if
$e'_{j_1}\in W^+$. In the following, let $e_{j_1}$ be the first edge term of $\mathbf{w}$ (i.e., the first edge term of $\mathbf{y}$), and $e_{j_*}$ be the last edge term of $\bw$.
We note $e_{j_1}\in E(\mathbf{y})\subset W^+$.

\smallskip
\noindent\textbf{(Case 1)}
Suppose that $\mathbf{x}$ is nontrivial.
Then $e_{j_*}$ is the last edge term of $\mathbf{x}$
and $\mathbf{x}+\mathbf{y}$ is a balanced section of $\bw$. Thus,
\begin{eqnarray}\label{eq:tau}
&&\tau(e_{j_*},v)=-\tau(e_{j_1},v).
\end{eqnarray}
We also note that $e_{j_*}\in E(\mathbf{x})\subset U^-$.
Then, where the second biconditional is from
\eqref{eq:tau},
\[e'_{j_1}\in U^+ ~~\Leftrightarrow ~ ~\tau(e_{j_*},v)=\tau(e'_{j_1},v)~~
\Leftrightarrow
~~-\tau(e_{j_1},v)=\tau(e'_{j_1},v)
~~\Leftrightarrow ~
~e'_{j_1}\in W^+.\]

\smallskip
\noindent\textbf{(Case 2)} Suppose that $\mathbf{x}$ is trivial. Then the first vertex term $v$ of $\bw$ is unbalanced, and so
\begin{eqnarray}\label{eq:tau2}&&\tau(e_{j_*},v)=\tau(e_{j_1},v).
\end{eqnarray}
We also note that $e_{j_*}\in E(\bw_{2r})\subset U^+$.
Then, where the second biconditional is from~\eqref{eq:tau2},
\[e'_{j_1}\in U^+ ~~\Leftrightarrow ~ ~\tau(e_{j_*},v)=-\tau(e'_{j_1},v)~~
\Leftrightarrow~~\tau(e_{j_1},v)=-\tau(e'_{j_1},v)~~\Leftrightarrow ~~e'_{j_1}\in W^+.\]
\end{proof}

\section{Proofs of the main results}\label{sec:proof}
\subsection{Proof of Theorem~\ref{thm:signed:primitive}}\label{sec:primitive}

\begin{proof}[Proof of Theorem~\ref{thm:signed:primitive}] The following is directly derived from Lemma~\ref{lem:product}.

\begin{claim}\label{lem:primitive:no_subwalk}
For an even-signed closed walk $\bw$ in a signed graph, if $B_{\bw}$ is primitive, then
$\bw$ has no proper nontrivial section that is an even-signed closed walk.
\end{claim}

Let $\bw:v_{i_1}e_{j_1}v_{i_2}\cdots v_{i_r}e_{j_r}v_{i_1}$ be an even-signed closed walk in a signed graph $(G,\tau)$.
First, we show the `only if' part.
Suppose that $B_{\bw}$ is primitive in $I_{(G,\tau)}$. Note that (ii) holds by Claim~\ref{lem:primitive:no_subwalk}, and so we will show (i).

\begin{claim}\label{lem:primitive:no_repetition}
Let $\bw=\bw_0+\bw_1+\bw_2+\bw_3$ be a nontrivial section-decomposition of $\bw$.
Then at least one of $\bw_0+\bw_1$ and $\bw_1+\bw_2$ is not a closed walk.
\end{claim}
\begin{proof}Suppose to contrary that each of $\bw_0+\bw_1$ and $\bw_1+\bw_2$ is a closed walk.
First, we will show that both $\bw_3+\bw_0$ and $\bw_3+\bw_1^{-1}$ are odd-signed closed walks in $(G,\tau)$.
Let $v$ be the first vertex term of $\bw_0$ and $u$ be the first vertex term of $\bw_1$.
Then $\bw_3$ is a $(u,v)$-walk, and
each of $\bw_0$  and $\bw^{-1}_1$ is a $(v,u)$-walk.
Hence, both $\bw_3+\bw_0$ and $\bw_3+\bw_1^{-1}$ are closed walks.
Then $\bw_3+\bw_0$ is clearly a proper closed section of $\bw$, and so it is odd-signed by Claim~\ref{lem:primitive:no_subwalk}.
In a closed walk $\bw^\ast=\bw_1^{-1}+\bw_0^{-1}+\bw_2+\bw_3$,
note that $\bw_3+\bw_1^{-1}$ is a proper closed section.
By Lemma~\ref{lem:twowalks:reverse}, $\bw^\ast$ is also an even closed walk such that $B_{\bw}=B_{\bw^\ast}$, and so $B_{\bw^\ast}$ is also primitive.
It follows from Claim~\ref{lem:primitive:no_subwalk} that $\bw_3+\bw_1^{-1}$ is odd-signed in $(G,\tau)$.
Then
\[\mu(\bw_0+\bw_1)=\mu(\bw_0+(\bw^{-1}_1)^{-1})=
\mu(\bw_3+\bw_0)\mu(\bw_3+\bw_1^{-1})=1,\]
where the second equality is from Lemma~\ref{lem:twowalks:basic} and the last one is from the fact that both  $\bw_3+\bw_0$ and $\bw_3+\bw_1^{-1}$ are odd-signed.
Hence, $\bw_0+\bw_1$ is a proper even-signed closed section of $\bw$, a contradiction to  Claim~\ref{lem:primitive:no_subwalk}.
\end{proof}

\begin{claim}\label{claim:repeat:cutvertex}
Let  $v$ be a vertex repeated in $\bw$.
Then $v$ is repeated exactly twice in $\bw$ and is a cut vertex of $[\bw]$ such that $[\bw]-v$ has exactly two connected components and each block of $[\bw]$ is a cycle of length at least two.
\end{claim}

\begin{proof}
We may assume that the first vertex term is $v$.
First, we claim that there is no vertex repeated more than twice.
Suppose that there are $k$ and $\ell$ such that $1<k<\ell \le r$ and $v=v_{i_{k}}=v_{i_{\ell}}$.
Let $\bw=\bw_0+\bw_1+\bw_2$ where each $\bw_t$ is a nontrivial closed walk whose first vertex term  is $v$. Then
$\mu(\bw_0)\mu(\bw_1)\mu(\bw_2)=\mu(\bw)=1$ by Lemma~\ref{eq:sign:product:closed},   which implies that at least one of $\bw_0$, $\bw_1$, and  $\bw_2$ is an even-signed closed walk in $(G,\tau)$, and say $\bw_0$.
This contradicts to Claim~\ref{lem:primitive:no_subwalk}.
Hence, $v$ is repeated twice and so we let $\bw=\bw_0+\bw_1$ where $\bw_t$ is a nontrivial closed walk whose first vertex term is $v$.

By Claim~\ref{lem:primitive:no_repetition}, every internal vertex term $u$ of $\bw_0$ is not appeared in $\bw_1$ at all. Hence $v$ is a cut vertex of $[\bw]$. Moreover, since $v$ is repeated twice, $[\bw]-v$ has exactly two connected components and every block of $[\bw]$ containing $v$  contains exactly two edges incident to $v$.
Thus, each block is a cycle.
\end{proof}

By Claim~\ref{claim:repeat:cutvertex}, it is clear that
every vertex belongs to at most two blocks, which implies (i).

\smallskip

We show the `if' part. Suppose that an even-signed closed walk $\bw$ in $(G,\tau)$ satisfies (i) and (ii).

\begin{claim}\label{claim:dividing:one}
For each edge $e\in E(\bw)$, $e$ divides exactly one of $B_{\bw}^+$ and $B_{\bw}^-$.
\end{claim}
\begin{proof}
Suppose to contrary that an edge $e\in E(\bw)$ divides both $B_{\bw}^+$ and $B_{\bw}^-$.
Then $e$ is repeated in $\bw$ and so $e$ is on a cycle of length two in $[\bw]$ by (i).
Moreover,  $\bw$ has at least two balanced sections, and let $\bw_0+\cdots+\bw_{2k-1}$ be a balanced section-decomposition of $\bw$.
Then we may assume that both $\bw_0$ and $\bw_{2i-1}$ contain $e$ for some $i\in [k]$.
Then there are section-decompositions $\bw_0=\mathbf{x}_0+\mathbf{y}_0$ and $\bw_{2i-1}=\mathbf{x}_{2i-1}+\mathbf{y}_{2i-1}$ such that
the first edge term of $\mathbf{y}_0$ and the last edge term of $\mathbf{x}_{2i-1}$ are $e$.
Consider the section $\bw'$ of $\bw$ so that the first and the last edge terms are $e$.
Then by the structure of $[\bw]$ from (i) and (ii), $\bw'$ is a closed walk
 and
$\bw'=\mathbf{y}_0+\mathbf{w}_1+\cdots+\mathbf{w}_{2i-2}+\mathbf{x}_{2i-1}$ is a balanced section-decomposition ($\mathbf{w}_1+\cdots+\mathbf{w}_{2i-2}$ is dropped if $i=1$). Hence, $\bw'$ is an even-signed closed section of $\bw$, a contradiction to (ii).
\end{proof}
Let $A=A(G,\tau)$.
By Claim~\ref{claim:dividing:one}, from the same way in (the proof of) Observation~\ref{obervation}, we can find an integer vector $\mathbb{b}=(b_e)_{e\in E(G)}$ such that $\mathbb{e}^{\mathbb{b}^+}-\mathbb{e}^{\mathbb{b}^-}=B_{\bw}$ and $G_{\mathbb{b}}=[\bw]$ ($G_{\mathbb{b}}$ is the multigraph in Proposition~\ref{rmk:basic:binomial}).
Suppose to contrary that $B_{\bw}$ is not primitive.
Then there is a binomial $\mathbb{e}^{\mathbb{c}^+}- \mathbb{e}^{\mathbb{c}^-}$ in $I_{(G,\tau)}$ (for some $\mathbb{c}=(c_e)_{e\in E(G)}$, other than $\mathbb{b}$) such that $\mathbb{e}^{\mathbb{c}^+}| \mathbb{e}^{\mathbb{b}^+}$, $\mathbb{e}^{\mathbb{c}^-}| \mathbb{e}^{\mathbb{b}^+}$.
It also holds $A\mathbb{c}=\mathbb{0}$. By proposition~\ref{rmk:basic:binomial},
each connected component of $G_{\mathbb{c}}$ has an even-signed Eulerian.
Now consider two multigraphs $G_{\mathbb{b}}$ and $G_{\mathbb{c}}$. Note that $G_{\mathbb{c}}$ is a proper subgraph of $G_{\mathbb{b}}$.
By the condition (i) on $G_{\mathbb{b}}$, each block $B$ of $G_{\mathbb{b}}$ is a cycle and so each block of $G_{\mathbb{c}}$ is also a block of $G_{\mathbb{b}}$.
Thus, $G_{\mathbb{c}}$ is made by taking some blocks of $G_{\mathbb{b}}$.

Let $\mathbb{d}=(d_e)_{e\in E(G)}$ be a vector such that $d_e=b_e-c_e$ for every edge $e$.
By definition,
\[  \mathbb{e}^{\mathbb{d}^+}=\frac{\mathbb{e}^{\mathbb{b}^+}}{\mathbb{e}^{\mathbb{c}^+}}\qquad \text{and}\qquad
 \mathbb{e}^{\mathbb{d}^-}=\frac{\mathbb{e}^{\mathbb{b}^-}}{\mathbb{e}^{\mathbb{c}^-}},\]
and $G_{\mathbb{d}}$ is the graph obtained from $G_{\mathbb{b}}$ by deleting the edges of blocks of $G_{\mathbb{c}}$. Take a nontrivial connected component $D$ of $G_{\mathbb{d}}$.
Since $A{\mathbb{d}}=\mathbb{0}$, by Proposition~\ref{rmk:basic:binomial}, $D$ has an even-signed Eulerian  $\mathbf{w}_D$.
However, $\mathbf{w}_D$ is a nontrivial section of $\bw$, which is an even-signed closed walk in $(G,\tau)$. This is a contradiction to (ii).
\end{proof}

\subsection{Proof of Theorem~\ref{thm:characterization:CI}}\label{sec:ci}

We often use the fact that a graph in $\mathcal{G}^{cis}$ satisfies all statements in Theorem~\ref{graph:useful:facts}, since $\mathcal{G}^{cis}\subset \mathcal{G}^{ci}$.

\begin{proof}[Proof of Theorem~\ref{thm:characterization:CI}]
We show the `only if' part first.
Suppose to contrary that $G$ is a $2$-connected graph in $\mathcal{G}^{cis}$, none of (G1)-(G5) in Theorem~\ref{thm:characterization:CI}.
Since $\mathcal{G}^{cis} \subset \mathcal{G}^{cio}$, by Theorem~\ref{thm:CI:digraph}(i) and Observation~\ref{lem:K4} it follows that $G$ is constructed by clique sums of cycles.
Note that since $G$ is 2-connected, 1-clique sum cannot be done to make $G$.
Thus, $G$ is constructed by $2$-clique sums of cycles.
By Corollary~\ref{cor:sumevencycle}, every induced subgraph which is constructed by $2$-clique sums of cycles belongs to $\mathcal{G}^{cis}$.
Not to be (G1) or (G2), $G$ is constructed by clique sums of at least three cycles.

\begin{claim}\label{claim:edge_two_cycle}
For each edge $e$, there are at most two induced cycles containing $e$.
\end{claim}
\begin{proof}[Proof of Claim~\ref{claim:edge_two_cycle}]
Suppose that there are three induced cycles $C^{(1)}$, $C^{(2)}$, and $C^{(3)}$ of $G$, containing the edge $e$.
Let $H=G[V(C^{(1)})\cup V(C^{(2)})\cup V(C^{(3)})]$.
Since $H$ is constructed by 2-clique sums of cycles,
$H\in \mathcal{G}^{cis}$ by Corollary~\ref{cor:sumevencycle}. Moreover, all vertices of $H$ except the endpoints of $e$ have degree two in the graph $H$.
By Proposition~\ref{lem:that_lemma}(i), by contracting ears of $H$, we obtain a graph $H^*\in \mathcal{G}^{cis}$, which is a $2$-clique sum of three triangles at one edge. But $H^*$ contains $K_{2,3}$, a contradiction to  Theorem~\ref{graph:useful:facts}(ii).
\end{proof}

\begin{claim}\label{claim:cycle_share_two}
An induced cycle of $G$ shares an edge with at most two induced cycles.
\end{claim}

\begin{proof} For an induced cycle $C$ of $G$, suppose that there are three induced cycles $C^{(1)}$, $C^{(2)}$, $C^{(3)}$ of $G$, each of which shares an edge with $C$.
Then for each $i\in[3]$ there is a unique edge $e_i$ which belongs to both  $C$ and $C^{(i)}$.
By Claim~\ref{claim:edge_two_cycle}, $e_1$, $e_2$, $e_3$ are distinct.
Now let $H=G[V(C)\cup V(C^{(1)})\cup V(C^{(2)})\cup V(C^{(3)})]$, and then $H\in\mathcal{G}^{cis}$ by Corollary~\ref{cor:sumevencycle}.
In addition, all vertices of $H$, except the endpoints of $e_i$'s, have degree two  in the graph $H$.
By Proposition~\ref{lem:that_lemma}(i), by contracting ears of $H$, we obtain a graph $H^*\in \mathcal{G}^{cis}$. Note that  $H^*$ is $2$-connected with three triangles. By Theorem~\ref{graph:useful:facts}(iii) and (iv), it follows that $H^*$ must be the graph $G_0$ in Figure~\ref{fig:EXAMPLE2}. By Example~\ref{ex}, $G_0 \not\in \mathcal{G}^{ci}$ and so $G_0\not\in \mathcal{G}^{cis}$, a contradiction.
\end{proof}

By Claims~\ref{claim:edge_two_cycle}~and~\ref{claim:cycle_share_two},
there are induced cycles $C^{(1)}$, $C^{(2)}$, $\ldots$, $C^{(m)}$ (for some $m\ge 3$) such that $G$ is constructed by $2$-clique sums of those $m$ cycles, where for each $i\in [m-1]$, $C^{(i)}$ and $C^{({i+1})}$ share an edge $e_i$. Note that $e_1$, $\ldots$, $e_{m-1}$ are distinct.
For each $i\in [m-2]$ and $\ell \in\{2,\ldots,m-i+1\}$, let $H^{(i)}_\ell=G[V(C^{(i)})\cup V(C^{(i+1)})\cup \cdots \cup V(C^{(i+\ell-1)})]$.
Note that $H_\ell^{(i)}\in \mathcal{G}^{cis}$ by Corollary~\ref{cor:sumevencycle}.  In addition,  Proposition~\ref{lem:that_lemma}(i) says that by contracting two ears of $H_\ell^{(i)}$ lying on the cycles $C^{(i)}$ and $C^{(i+\ell-1)}$, we obtain a graph $F^{(i)}_\ell$ in $\mathcal{G}^{cis}$.
In $F^{(i)}_\ell$, the cycles corresponding to $C^{(i)}$ and $C^{(i+\ell-1)}$ are triangles.

\begin{claim}\label{claim:vertex_cycle}
Let $i\in[m-2]$.
Then $C^{(i+1)}$ has length at most four.
Moreover, if $C^{({i})}$, $C^{(i+1)}$, and $C^{(i+2)}$ share a vertex $v$, then
$C^{(i+1)}$ is a triangle and there is no more induced cycle containing $v$.
\end{claim}
\begin{proof} We firstly show the `moreover' part.
Suppose that $C^{(i)}$, $C^{(i+1)}$, and $C^{(i+2)}$ share a vertex $v$.
If $C^{(i+1)}$ is not a triangle, then by contracting ears of $F^{(i)}_3$ properly, we can obtain $G_1 $ in Figure~\ref{fig:AB}, and note that
$G_1\not\in \mathcal{G}^{ci}$ by  Theorem~\ref{graph:useful:facts}(iii).
Hence, $C^{(i+1)}$ is a triangle.
\begin{figure}[ht!]
  \centering
  \includegraphics[width=7.5cm,page=11]{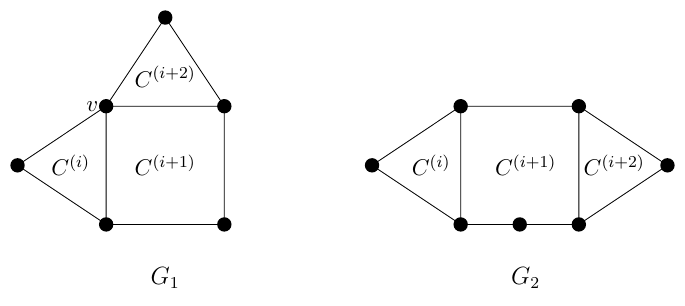}\\\caption{Some graphs not in $\mathcal{G}^{cis}$}\label{fig:AB}
\end{figure}
Suppose that there is another induced cycle $C^{(j)}$ containing the vertex $v$. Since $G$ is $2$-connected, we may assume that $j=i+3$. By the above argument, both $C^{(i+1)}$ and $C^{(i+2)}$ are triangles.
Then $F^{(i)}_4$ has four triangles and so $|V(F^{(i)}_4)|= 6$ and $|E(F^{(i)}_4)|=9$.
By deleting the vertex $v$, it becomes a bipartite graph, and so $\sum_{x} b(F^{(i)}_4-x)\ge 1$. Applying  Theorem~\ref{graph:useful:facts}(i),
we have $2|E(F^{(i)}_4)|< 3|V(F^{(i)}_4)|$, a contradiction.

It remains to show that $C^{(i+1)}$ has length at most 4.
Suppose that $C^{(i+1)}$ has length at least 5. Then by the previous argument,   $C^{(i)}$, $C^{(i+1)}$, and $C^{(i+2)}$ do not share one common vertex. By contracting ears of $F^{(i)}_3$ properly, we obtain $G_2$ in Figure~\ref{fig:AB}.
By Theorem~\ref{graph:useful:facts}(iv), $G_2\not\in\mathcal{G}^{ci}$,
a contradiction.
\end{proof}

Suppose that $C^{(2)}$ is a cycle of at least length 4.
By Claim~\ref{claim:vertex_cycle}, $C^{(2)}$ has length 4.
By the moreover part of Claim~\ref{claim:vertex_cycle}, $C^{(1)}$ and  $C^{(3)}$ do not share a vertex.
Not to be (G4), $m\ge 4$, and so consider $F^{(1)}_4$.
Then the triangles corresponding to $C^{(1)}$ and $C^{(4)}$ in $F^{(1)}_4$ are disjoint by Claim~\ref{claim:vertex_cycle}.
Then there are no two disjoint edges connecting those two triangles,
which is a contradiction to  Theorem~\ref{graph:useful:facts}(iv).
Suppose that $C^{(2)}$ is a triangle.
By Claim~\ref{claim:vertex_cycle}, not to be  (G5), $m\ge 5$.
Consider  $F^{(1)}_5$ and then
there are no two disjoint edges connecting two triangles
corresponding to $C^{(1)}$ and $C^{(5)}$  in $F^{(1)}_5$,
a contradiction to Theorem~\ref{graph:useful:facts}(iv).

\smallskip

Now we prove the `if' part of Theorem~\ref{thm:characterization:CI}.
By Proposition~\ref{lem:that_lemma}(ii),
if (G$i{}^{\prime}$) is in $\mathcal{G}^{cis}$ then (G$i$) is in  $\mathcal{G}^{cis}$, where (G${1}^{\prime}$)-(G${5}^{\prime}$) are in Figure~\ref{fig:Last6}.
Let $G$ be one of (G${1}^{\prime}$)-(G${5}^{\prime}$), and $\tau$ be its sign.
\begin{figure}[ht!]
  \centering
  \includegraphics[width=12.5cm,page=12]{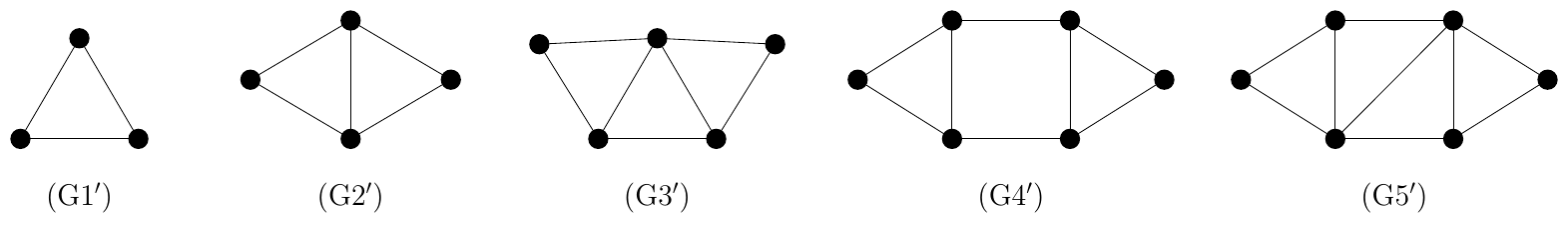}\\
  \caption{Graphs (G${1}^{\prime}$)-(G${5}^{\prime}$)}\label{fig:Last6}
\end{figure}

\smallskip
\noindent\textbf{(G${1}^{\prime}$) and (G${2}^{\prime}$)}
It is trivial that the graph (G${1}^{\prime}$) is in $\mathcal{G}^{cis}$,
since either $I_{(G,\tau)}=\{0\}$ and the triangle $\bw$ is odd-signed in $(G,\tau)$,  or $I_{(G,\tau)}=\left< B_{\bw}\right>$ and $\bw$ is even-signed in $(G,\tau)$.
Suppose that $G$ is (G${2}^{\prime}$).
Let $\bw$ and $\bw'$ be two triangles of $G$.
If one of  $\bw$ and $\bw'$ is even-signed in $(G,\tau)$, then $I_{(G,\tau)}$ is a complete intersection by Proposition~\ref{lem:sumevencycle}.
Suppose that $\bw$ and $\bw'$ are odd-signed in $(G,\tau)$.
Then $(G,\tau)$ has only one primitive walk, by Theorem~\ref{thm:signed:primitive}, which is the cycle $\bw''$ of length four.
Thus, $I_{(G,\tau)}=\left<B_{\bw''} \right>$, and so $I_{(G,\tau)}$ is a complete intersection.

\smallskip

\noindent\textbf{(G${3}^{\prime}$)} Suppose that $G$ is (G${3}^{\prime}$).
If one of triangles having a vertex of degree two is even-signed in $(G,\tau)$, then $I_{(G,\tau)}$ is a complete intersection by
 Proposition~\ref{lem:sumevencycle}, since we already show that the graph in (G${2}^{\prime}$) is in $\mathcal{G}^{cis}$.
Hence, suppose that both  triangles having a vertex of degree two are odd-signed in $(G,\tau)$. We will find $2(=r(G,\tau))$ binomials which generate $I_{(G,\tau)}$.
\begin{figure}[ht!]
  \centering
  \includegraphics[width=13cm,page=13]{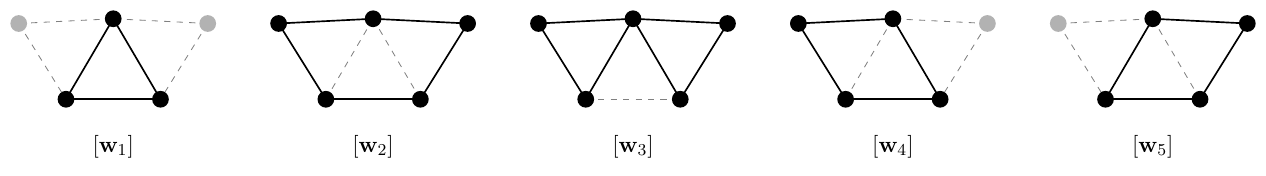}\\\caption{All possible primitive walks in $(G,\tau)$ when $G$ is (G$3^{\prime}$)}\label{fig:2connCIS4-1}
\end{figure}
By Theorem~\ref{thm:signed:primitive}, the set of primitive walks is a subset of $\{\bw_1,\ldots,\bw_5\}$.
If the triangle $\bw_1$ is even-signed in $(G,\tau)$, then $G$ has only three  primitive walks $\bw_1$, $\bw_2$ and $\bw_3$, and $B_{\bw_3}\in \left<B_{\bw_1},B_{\bw_2}\right>$ by Corollary~\ref{rmk2:Lem4_1}.
If $\bw_1$ is odd-signed in $(G,\tau)$, then $G$ has only  three primitive walks $\bw_3$, $\bw_4$ and $\bw_5$, and $B_{\bw_3}\in \left<B_{\bw_4},B_{\bw_5}\right>$ by Corollary~\ref{rmk2:Lem4_1}.

\smallskip

\noindent\textbf{(G${4}^{\prime}$)}  Suppose that $G$ is (G${4}^{\prime}$).
Similar to previous case, by Proposition~\ref{lem:sumevencycle}, we may assume that two  triangles are odd-signed in $(G,\tau)$. Then we will find $2(=r({G,\tau}))$ binomials which generate $I_{(G,\tau)}$.
Note that the six walks defined as Figure~\ref{fig:2connCIS5-1} are all possible primitive walks.
\begin{figure}[ht!]
  \centering
  \includegraphics[width=16cm,page=14]{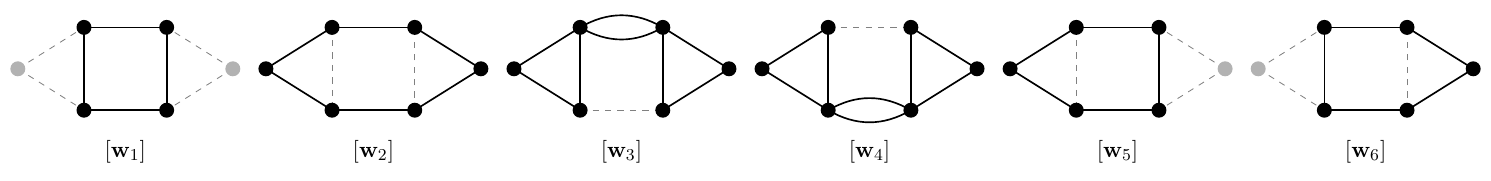}\\
   \caption{All possible primitive walks in $(G,\tau)$ when $G$ is (G$4^{\prime}$)}\label{fig:2connCIS5-1}
\end{figure}

If $\bw_1$ is even-signed in $(G,\tau)$, then   $(G,\tau)$ has only four primitive walks $\bw_1$, $\bw_2$, $\bw_3$ and $\bw_4$, and $B_{\bw_3},B_{\bw_4}\in \left<B_{\bw_1},B_{\bw_2}\right>$ by Corollary~\ref{rmk2:Lem4_1}.
If $\bw_1$ is odd-signed in $(G,\tau)$, then $(G,\tau)$ has only four primitive walks $\bw_3$, $\bw_4$, $\bw_5$ and $\bw_6$, and $B_{\bw_3},B_{\bw_4}\in \left<B_{\bw_5},B_{\bw_6}\right>$ by Corollary~\ref{rmk2:Lem4_1}.

\smallskip

\noindent\textbf{(G${5}^{\prime}$)}  Suppose that $G$ is  (G${5}^{\prime}$).
Similar to previous case, by Proposition~\ref{lem:sumevencycle}, we  may assume that two triangles having a vertex of degree two are odd-signed in $(G,\tau)$.
We will find $3(=r({G,\tau}))$ binomials which generate $I_{(G,\tau)}$.
Consider  closed walks $\mathbf{a}$, $\mathbf{b}$, $\mathbf{x}_1$, $\mathbf{x}_2$, and $\mathbf{x}_3$, defined as Figure~\ref{fig:G5}.

\begin{figure}[ht!]
  \centering
  \includegraphics[width=15cm,page=15]{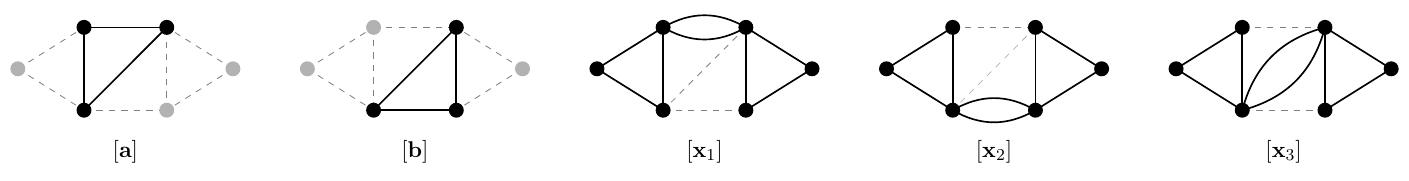}\\
   \caption{Five closed walks in $G$ when $G$ is (G$5^{\prime}$), where $\mathbf{x}_1\sim \mathbf{x}_3$ are even-signed in $(G,\tau)$}\label{fig:G5}
\end{figure}
We consider cases according to $\mu(\mathbf{a})$ and $\mu(\mathbf{b})$,
and then, in each case we will define six walks $\bw_1\sim \bw_6$ as Figure~\ref{fig:G55}.\footnote{The case where $\mu(\mathbf{a})=1$ and $\mu(\mathbf{b})=-1$ is similar to the second case of Figure~\ref{fig:G55}.}
Then $(G,\tau)$ has only 9 primitive walks, $\bw_1$, $\ldots$, $\bw_6$, $\mathbf{x}_1$, $\mathbf{x}_2$, and $\mathbf{x}_3$, and then   $B_{\bw_1},B_{\bw_2},B_{\bw_3}$ generate $I_{(G,\tau)}$, since
 it follows from Corollary~\ref{rmk2:Lem4_1} that
\[B_{\bw_4}\in \left<B_{\bw_1},B_{\bw_2}\right>,\ \ B_{\bw_5}\in \left<B_{\bw_1},B_{\bw_3}\right>,\ \ B_{\bw_6}\in \left<B_{\bw_2},B_{\bw_3}\right>,
\ \ B_{\bx_1},B_{\bx_2}\in \left<B_{\bw_3},B_{\bw_4}\right>,\ \ B_{\bx_3}\in \left<B_{\bw_1},B_{\bw_6}\right>.\]

\begin{figure}[h!]
\centering
\begin{tabular}{c|c}
\toprule
  $\mu(\mathbf{a})$, $\mu(\mathbf{b})$  & Primitive walks $\bw_1\sim\bw_6$ \\ \hline \hline
\multirow{2}{*}
{\!\!\!\!$\mu(\mathbf{a})=\mu(\mathbf{b})=1$\!\!\!\!}&\\
&{\!\!\!\includegraphics[width=12cm,page=16]{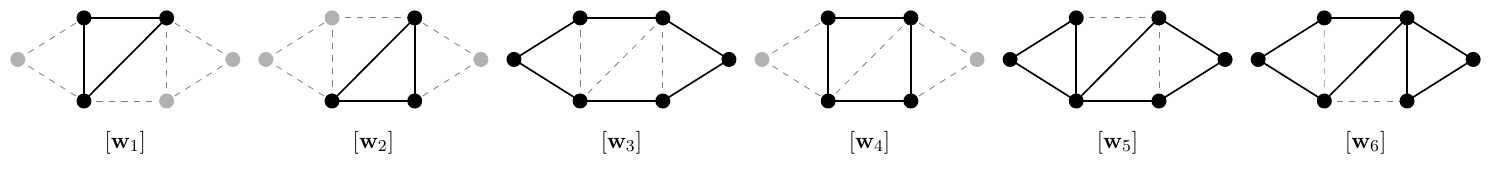}} \\ \hline
\multirow{2}{*}{\!\!$\mu(\mathbf{a})=-1$,
$\mu(\mathbf{b})=1$\!\!\!}     &  \\
&{\!\!\!\includegraphics[width=12cm,page=17]{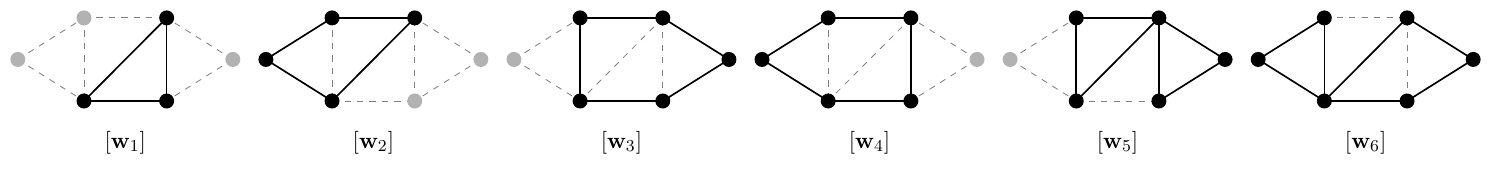}} \\ \hline
\multirow{2}{*}{$\mu(\mathbf{a})=\mu(\mathbf{b})=-1$} &  \\
&{\!\!\!\includegraphics[width=12cm,page=18]{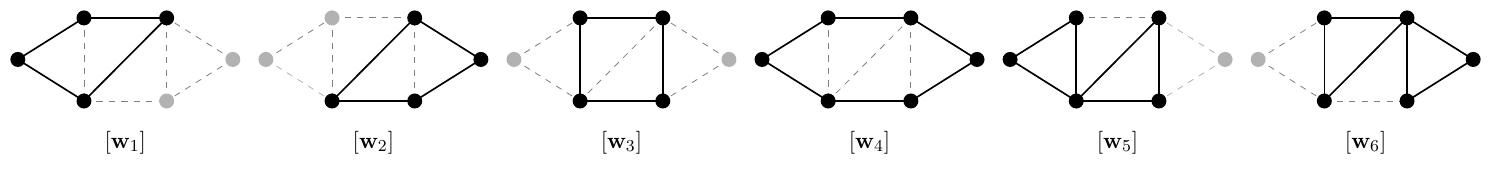}} \\
\bottomrule
\end{tabular}\vspace{-0.2cm}
\caption{Six primitive walks when $G$ is (G$5^{\prime})$}\label{fig:G55}
\end{figure}

\end{proof}

\subsection{Proof of Theorem~\ref{thm:cis:general:graph}}\label{sec:general}

First, we note that each of $G_1\sim G_8$ in Figure~\ref{fig:graphs:not:cis} does not belong to $\mathcal{G}^{ci}$
by Algorithm~\ref{thm:CI_Graphs}.\footnote{For the graph $G_3$, \cite[Example 4.10]{BGR2015}, it was shown that the toric ideal $I_G$ is not a complete intersection by using the algorithm.
Fix $G_i$ for some in $i\in \{1,2,4,5,6,7,8\}$.
Since $G_i$ is not bipartite,
we need at least  $|E(G_i)|-|V(G_i)|$ nontrivial binomials to generate the ideal $I_{G_i}$.
Let $v$ be a vertex of degree two on the block isomorphic to $K_3$. Note that $b(G_i-v)=b(G_i)=0$, and
let us apply Algorithm~\ref{thm:CI_Graphs} starting from the vertex $v$.
If $i=4$, then there is no such closed walk $\bw$ of even length with $V(\bw)=W$ and so the algorithm returns \textsc{False}.
Otherwise, the binomial associated with any shortest closed walk $\bw$ of even length with $V(\bw)=W$ is trivial. In the remaining process, we consider the graph $G'_i=G_i-v$, which is not bipartite. Since $|E(G'_i)|-|V(G'_i)|=|E(G_i)|-|V(G_i)|-1$,
at most $|E(G_i)|-|V(G_i)|-1$ nontrivial binomials are obtained through the algorithm, and those cannot generate $I_{G_i}$. Consequently, the algorithm returns \textsc{False}.} Therefore, those eight graphs are not in $\mathcal{G}^{cis}$.

\begin{figure}[ht!]
  \centering
  \includegraphics[width=12cm,page=19]{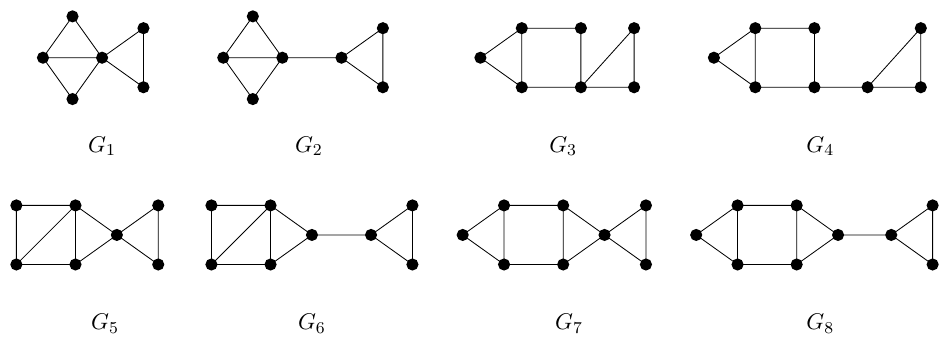}\\
   \caption{Eight graphs $G_1\sim G_8$ whose toric ideals are not complete intersections}  \label{fig:graphs:not:cis}
\end{figure}

\begin{proof}[Proof of Theorem~\ref{thm:cis:general:graph}]
It is sufficient to consider only connected graphs. We first show the `only if' part.
Suppose that $G$ is a connected graph in $\mathcal{G}^{cis}$.
For a nonedge block $F$ of $G$, $F$ is one of (G1)-(G5) by Proposition~\ref{lem:each:block} and Theorem~\ref{thm:characterization:CI}.
If $G$ has at most one nonedge block, then (i) or (ii) holds.
Now suppose that $G$ has at least two nonedge blocks.

\begin{claim}\label{claim:B1:C1}
Suppose that $F_1$ and $F_2$ are two nonedge blocks of $G$ such that $F_1$ is not {\rm(G1)}.
Let $\mathbf{p}$ be a shortest path from a vertex of $F_1$ to a vertex of $F_2$,
and
$C^{(i)}$ an induced cycle of $F_i$ having a vertex of $\mathbf{p}$.
Then $F_1$ is {\rm(G2)}, the cycle $C^{(1)}$ is a triangle, and $\deg_{F_1}(v_1)=2$, where $v_1=V(F_1)\cap V(\mathbf{p})$.
\end{claim}

\begin{proof}
Since $F_1$ is not (G1), we can take another induced cycle $C^{(0)}$ of $F_1$ which shares an edge with $C^{(1)}$.
Let $H=G[V(C^{(0)})\cup V(C^{(1)})\cup V(C^{(2)})\cup V(\mathbf{p})]$.
Note that $H$ is in $\mathcal{G}^{cis}$ by Corollary~\ref{cor:sumevencycle}.
Using Proposition~\ref{lem:that_lemma}(i),
by contracting ears, we obtain a graph $H^*$ in $\mathcal{G}^{cis}$.
Let $C^{(i)}_*$ and $\mathbf{p}^*$ be the cycle and the path of $H^*$ corresponding to $C^{(i)}$ and $\mathbf{p}$, respectively. Note that $\mathbf{p}^*$ is a path of length at most one.
Since the graphs $G_1$ and $G_2$ in Figure~\ref{fig:graphs:not:cis} are not in $\mathcal{G}^{cis}$,
it follows that $v_1$ has degree two in
$H^*[ V(C^{(0)}_*)\cup V(C^{(1)}_*)]$.

From the fact that the graphs $G_3$ and $G_4$ in Figure~\ref{fig:graphs:not:cis} are not in $\mathcal{G}^{cis}$, together with Proposition~\ref{lem:that_lemma}(i), the cycle $C^{(1)}$ must be a triangle.
Similarly, from the fact that the graphs $G_5$, $G_6$, $G_7$, and $G_8$ in Figure~\ref{fig:graphs:not:cis} are not in $\mathcal{G}^{cis}$,
it follows that $F_1$ must be (G2), and therefore,
$\deg_{F_1}(v_1)=2$.
\end{proof}

By Claim~\ref{claim:B1:C1}, it is sufficient to show that $G$ has at most two nonedge blocks.
Suppose to contrary that $G$ has three nonedge blocks $F_1$, $F_2$, $F_3$.
Without loss of generality, we may assume that the distance between $F_1$ and $F_2$ is maximum among the distances between two of $F_1$, $F_2$ and $F_3$.
Let $\mathbf{p}$ be a shortest path between $F_1$ and $F_2$.
For each $i\in \{1,2\}$, we take an induced cycle $C^{(i)}$ of $F_i$ having a vertex of $\mathbf{p}$.
Let $H$ be a smallest induced connected subgraph of $G$ containing $V(C^{(1)})\cup V(C^{(2)})\cup V(\mathbf{p})\cup V(F_3)$. Note that $H\in\mathcal{G}^{cis}$ by Corollary~\ref{cor:sumevencycle}.
In addition, there is a vertex $w_i$  such that $\deg_{H}(w_i)=2$ and $w_i\in C^{(i)}$ for each $i=1,2$.
By Proposition~\ref{lem:that_lemma}(ii), by subdividing edges incident to $w_1$ and $w_2$ properly so that $C^{(1)}$ and $C^{(2)}$ become cycles of odd length, we can obtain a new graph $H'$ in $\mathcal{G}^{cis}$ with two non-bipartite blocks.
If $F_3$ is not bipartite,
then $H'$ has three non-bipartite blocks, and so
$H'\not\in\mathcal{G}^{ci}$ by  Theorem~\ref{graph:useful:facts}(v), a contradiction.
Thus, $F_3$ is bipartite.
Then we can find a vertex $w_3\in F_3$ with $\deg_{H'}(w_3)=2$ and let $H''$ be the graph obtained from $H'$ by subdividing an edge incident to $w_3$ once. Then $H''$ has three non-bipartite blocks and so $H'\not\in\mathcal{G}^{ci}$
 by Theorem~\ref{graph:useful:facts}(v).
On the other hand, $H'\in\mathcal{G}^{cis}$ by  Proposition~\ref{lem:that_lemma}(ii), and we reach a contradiction.

\smallskip

Now we show the `only if' part.
Note that for every sign $\tau$ of $G$,
any primitive walk in $(G,\tau)$ does not contain a pendent edge $e$
by Theorem~\ref{thm:signed:primitive}.
Thus, each graph satisfying (i) or (ii)  is in $\mathcal{G}^{cis}$ by Theorem~\ref{thm:characterization:CI}.
Consider graphs satisfying (iii).
Together with Proposition~\ref{lem:that_lemma}(ii), it is sufficient to show that each of the six graphs (H$1$)-(H$3$) and (H$1'$)-(H$3'$) in Figure~\ref{fig:reamin_six_graphs} is in $\mathcal{G}^{cis}$. We consider (H$1$)-(H$3$) first.
Let $G$ be one of (H${1}$)-(H${3}$), and $\tau$ be its sign.

\begin{figure}[ht!]
  \centering
  \includegraphics[width=10.5cm,page=20]{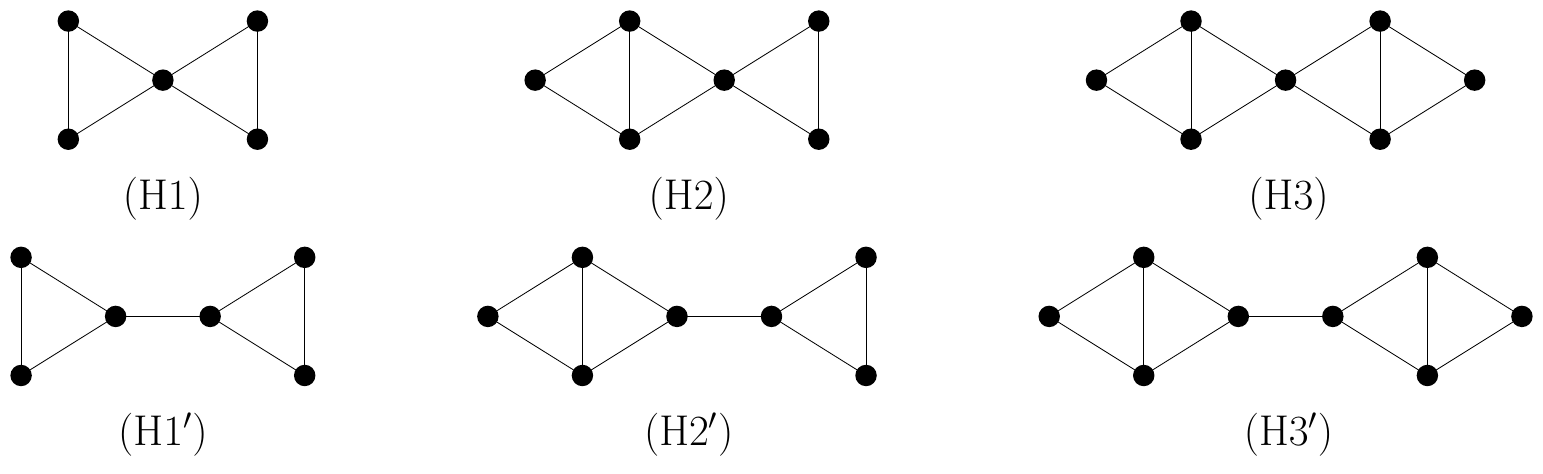}\\
   \caption{Graphs (H$1$)-(H$3$) and (H$1'$)-(H$3'$)}\label{fig:reamin_six_graphs}
\end{figure}

\noindent \textbf{(H$1$)} Suppose that $G$ is (H$1$). Let $\bw$ and $\bw'$ be two triangles in $G$. By Proposition~\ref{lem:sumevencycle}, it is sufficient to suppose that each of $\mathbf{\bw}$ and $\mathbf{\bw'}$ is odd-signed in $(G,\tau)$. Then there is only one primitive walk $\bw+\bw'$ and $r(G,\tau)=1$, which implies that $(G,\tau)$ is a complete intersection.

\smallskip

\noindent \textbf{(H$2$)} Suppose that $G$ is (H$2$). Let $\mathbf{a}$  be the triangle without a vertex of degree two. Since we already show that (G$2'$) in Figure \ref{fig:Last6} and (H$1$) are in $\mathcal{G}^{cis}$, by Proposition~\ref{lem:sumevencycle}, we may assume that the two triangles other than $\mathbf{a}$ are odd-signed in $(G,\tau)$. We will find $2(=r(G,\tau))$ binomials which generate $I_{(G,\tau)}$. By Theorem~\ref{thm:signed:primitive}, the six walks defined as Figure~\ref{fig:H22} are all possible primitive walks.

\begin{figure}[ht!]
  \centering
  \includegraphics[width=15cm,page=21]{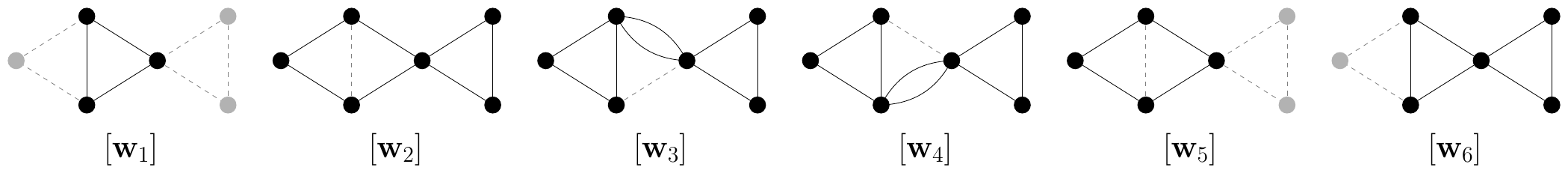}\\ \caption{All possible primitive walks in $(G,\tau)$ when $G$ is (H$2$)}\label{fig:H22}
\end{figure}

If $\mathbf{a}$ is even-signed in $(G,\tau)$, then $(G,\tau)$ has only four primitive walks $\bw_1$, $\bw_2$, $\bw_3$, and $\bw_4$, and $B_{\bw_3},B_{\bw_4} \in \left<B_{\bw_1},B_{\bw_2}\right>$ by Corollary \ref{rmk2:Lem4_1}.
If $\mathbf{a}$ is odd-signed in $(G,\tau)$,
then $(G,\tau)$ has only four primitive walks $\bw_3$, $\bw_4$, $\bw_5$, and $\bw_6$, and $B_{\bw_3},B_{\bw_4} \in \left<B_{\bw_5},B_{\bw_6}\right>$ by Corollary \ref{rmk2:Lem4_1}.

\smallskip

\noindent \textbf{(H$3$)} Suppose that $G$ is (H$3$).
Since we already show that (H$2$) is in  $\mathcal{G}^{cis}$,
by Proposition~\ref{lem:sumevencycle}, we may assume that both  triangles having a vertex of degree two are odd-signed in $(G,\tau)$.
Then we will find $3(=r(G,\tau))$ binomials which generate $I_{(G,\tau)}$. Consider six walks $\mathbf{a}$, $\mathbf{b}$, $\mathbf{x}_1\sim\mathbf{x}_4$ defined as Figure~\ref{fig:H3}.

\begin{figure}[h!]
  \centering
  \includegraphics[width=10cm,page=22]{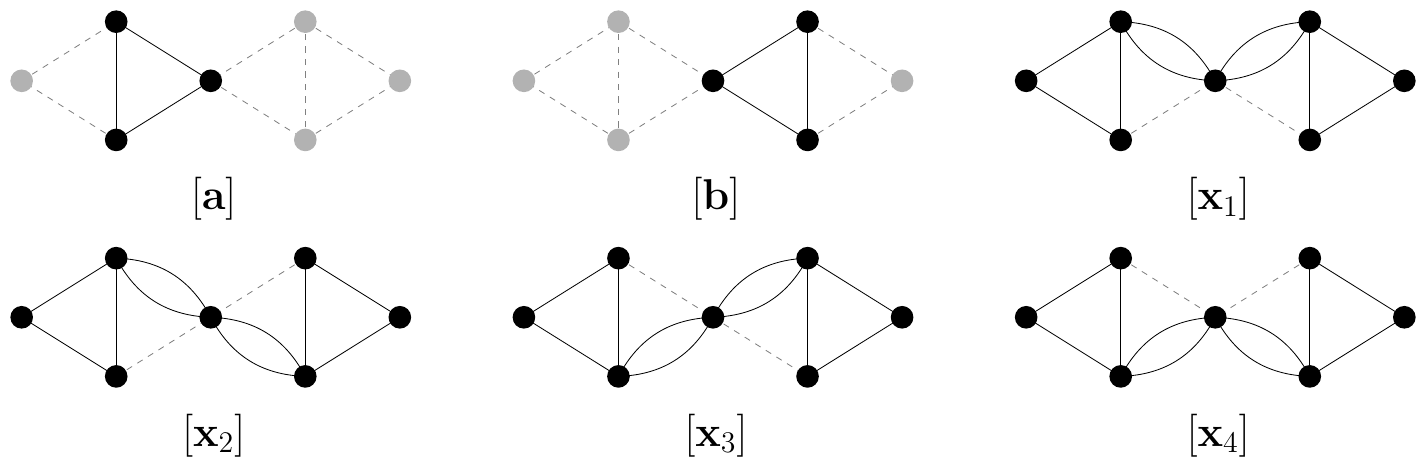}\\
   \caption{Six closed walks in $G$ when $G$ is (H$3$), where $\mathbf{x}_1\sim\mathbf{x}_4$ are even-signed in $(G,\tau)$}\label{fig:H3}
\end{figure}

 Now we divide cases according to $\mu(\mathbf{a})$ and $\mu(\mathbf{b})$. In each case, we define eight walks $\mathbf{w}_1\sim\mathbf{w}_8$ as Figure \ref{fig:H3333}.\footnote{The case where $\mu(\mathbf{a})=-1$ and $\mu(\mathbf{b})=1$ is similar to the second case of Figure \ref{fig:H3333}.} Then $(G,\tau)$ has only 12 primitive walks, $\mathbf{w}_1, \dots, \mathbf{w}_8$, $\mathbf{x}_1$,\dots, $\mathbf{x}_4$, and then  $B_{\bw_1},B_{\bw_2},B_{\bw_3}$ generate $I_{(G,\tau)}$, since it follows from Corollary~\ref{rmk2:Lem4_1} that
\begin{eqnarray*}
&&B_{\bw_4}\in \left<B_{\bw_1},B_{\bw_2}\right>,\quad B_{\bw_5},B_{\bw_6}\in \left<B_{\bw_1},B_{\bw_3}\right>,\quad B_{\bw_7}, B_{\bw_8}\in \left<B_{\bw_2},B_{\bw_3}\right>,\\
&& B_{\bx_1},B_{\bx_2}\in \left<B_{\bw_2},B_{\bw_6}\right>,\quad B_{\bx_3},B_{\bx_4}\in \left<B_{\bw_2},B_{\bw_5}\right>.
\end{eqnarray*}

\begin{figure}[h!]
\centering
\begin{tabular}{c|c}
\toprule
  $\mu(\mathbf{a})$, $\mu(\mathbf{b})$  & Primitive walks $\bw_1\sim\bw_8$ \\ \hline \hline
\multirow{2}{*}
{$\mu(\mathbf{a})=\mu(\mathbf{b})=1$}&\\
& {\includegraphics[width=12.2cm,page=23]{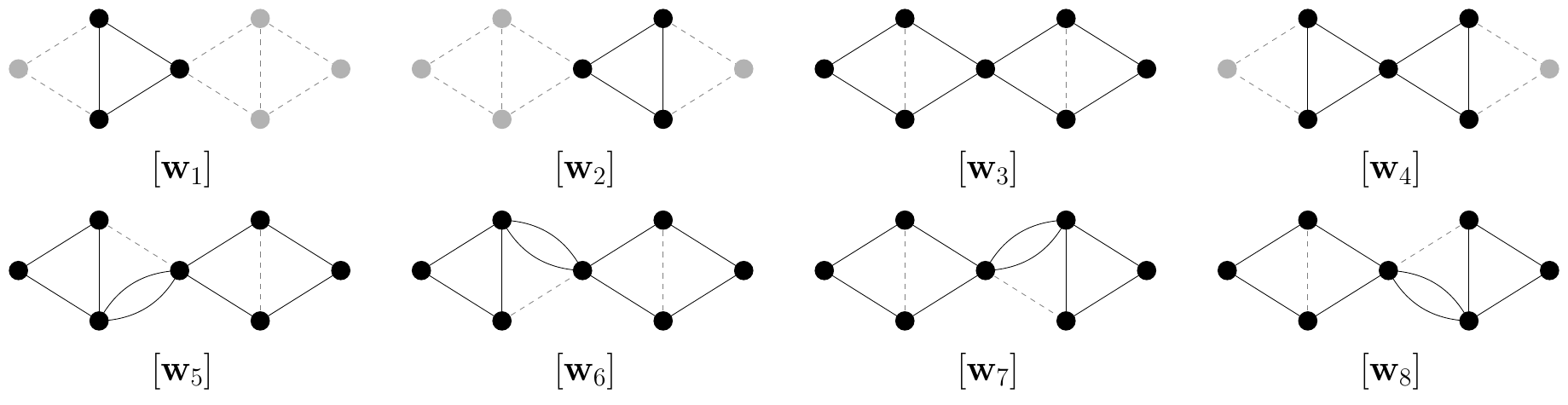}} \\ \hline
\multirow{2}{*}{$\mu(\mathbf{a})=1$,
$\mu(\mathbf{b})=-1$}     &  \\
& {\includegraphics[width=12.2cm,page=24]{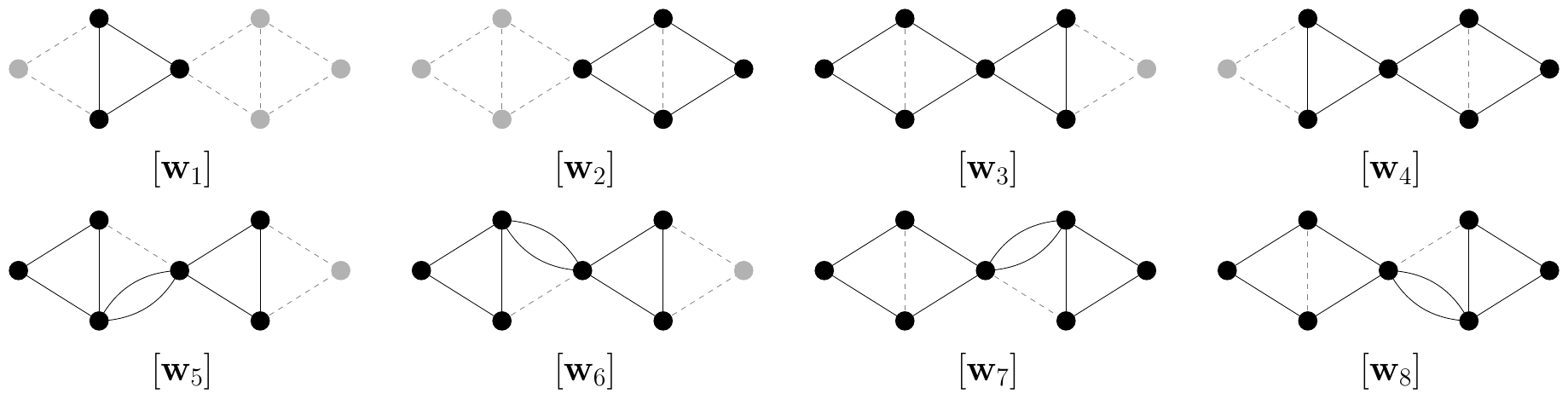}} \\ \hline
\multirow{2}{*}{$\mu(\mathbf{a})=\mu(\mathbf{b})=-1$} &  \\
& {\includegraphics[width=12.2cm,page=25]{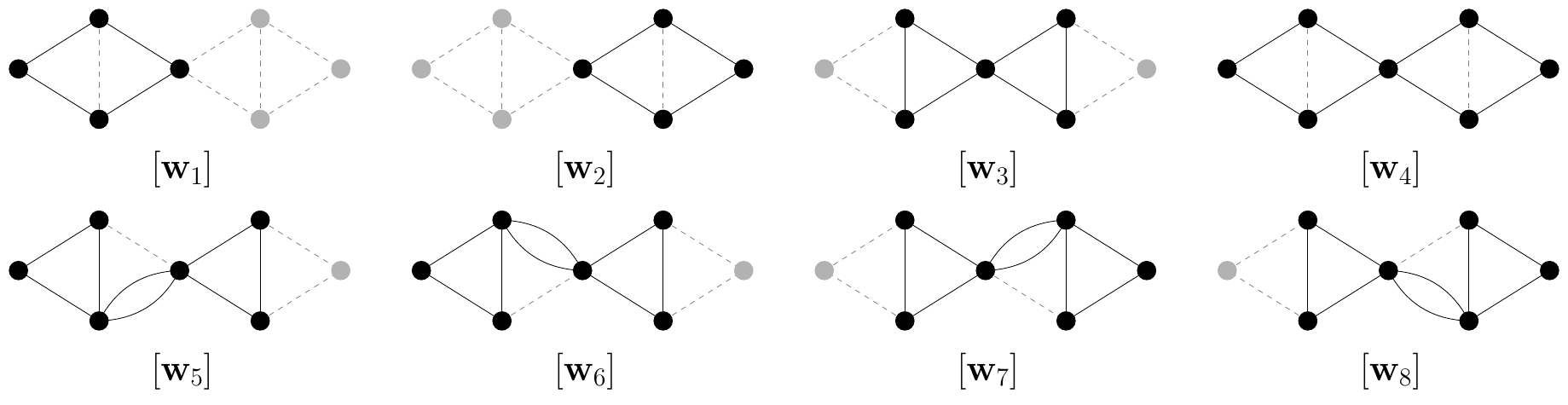}} \\
\bottomrule
\end{tabular}\vspace{-0.2cm}
\caption{Eight primitive walks when $G$ is (H$3$)}\label{fig:H3333}
\end{figure}

Let $(G,\tau)$ and $(G,\tau')$ be such that
$G$ is (H$i$), $G'$ is (H$i'$) ($i\in \{1,2,3\}$), and the sign coincides on a cycle.
Then $r(G,\tau)=r(G',\tau')$ and the primitive walks are also corresponding.
Similar to the argument of (H$i$), we can conclude that (H$i'$) is in $\mathcal{G}^{cis}$.
\end{proof}

\section*{Acknowledgement}
The authors would like to thank the anonymous reviewers for their insightful comments and suggestions.
JiSun Huh was supported by Basic Science Research Program through the National Research Foundation of Korea (NRF) funded by the Ministry of Science, ICT and Future Planning (NRF-2020R1C1C1A01008524).
Sangwook Kim was supported by Basic Science Research Program through the National Research Foundation of Korea(NRF) funded by the Ministry of Education
(NRF-2017R1D1A3B03031839).
Boram Park was supported by Basic Science Research Program through the National Research Foundation of Korea (NRF) funded by the Ministry of Science, ICT and Future Planning (NRF-2018R1C1B6003577).

\bibliographystyle{plain}

\begin{thebibliography}{10}

\bibitem{BG2015}
Isabel Bermejo and Ignacio Garc{\'\i}a-Marco.
\newblock Complete intersections in simplicial toric varieties.
\newblock {\em Journal of Symbolic Computation}, 68:265--286, 2015.

\bibitem{BGR2015}
Isabel Bermejo, Ignacio Garc{\'\i}a-Marco, and Enrique Reyes.
\newblock Graphs and complete intersection toric ideals.
\newblock {\em Journal of Algebra and its Applications}, 14(09):1540011, 2015.

\bibitem{biermann2017bounds}
Jennifer Biermann, Augustine O'Keefe, and Adam Van~Tuyl.
\newblock Bounds on the regularity of toric ideals of graphs.
\newblock {\em Advances in Applied Mathematics}, 85:84--102, 2017.

\bibitem{gap2015}
Alessio D'Al{\`\i}.
\newblock Toric ideals associated with gap-free graphs.
\newblock {\em Journal of Pure and Applied Algebra}, 219(9):3862--3872, 2015.

\bibitem{ES1996}
David Eisenbud and Bernd Sturmfels.
\newblock Binomial ideals.
\newblock {\em Duke Mathematical Journal}, 84(1):1--45, 1996.

\bibitem{Fulton}
William Fulton.
\newblock {\em Introduction to toric varieties}.
\newblock Princeton University Press, 1993.

\bibitem{gitler2017cio}
Isidoro Gitler, Enrique Reyes, and Juan~A Vega.
\newblock Cio and ring graphs: Deficiency and testing.
\newblock {\em Journal of Symbolic Computation}, 79:249--268, 2017.

\bibitem{GRV2013CI}
Isidoro Gitler, Enrique Reyes, and Juan~Antonio Vega.
\newblock Complete intersection toric ideals of oriented graphs and
  chorded-theta subgraphs.
\newblock {\em Journal of Algebraic Combinatorics}, 38(3):721--744, 2013.

\bibitem{gitler2010ring}
Isidoro Gitler, Enrique Reyes, and Rafael~H Villarreal.
\newblock Ring graphs and complete intersection toric ideals.
\newblock {\em Discrete Mathematics}, 310(3):430--441, 2010.

\bibitem{H1970}
J{\"u}rgen Herzog.
\newblock Generators and relations of abelian semigroups and semigroup rings.
\newblock {\em Manuscripta mathematica}, 3(2):175--193, 1970.

\bibitem{Hibi}
Takayuki Hibi.
\newblock {\em Gr{\"o}bner Bases}.
\newblock Springer, 2014.

\bibitem{kashiwabara2010toric}
Kenji Kashiwabara.
\newblock The toric ideal of a matroid of rank 3 is generated by quadrics.
\newblock {\em The Electronic Journal of Combinatorics}, 17(1):R28, 2010.

\bibitem{lason2014toric}
Micha{\l} Laso{\'n} and Mateusz Micha{\l}ek.
\newblock On the toric ideal of a matroid.
\newblock {\em Advances in Mathematics}, 259:1--12, 2014.

\bibitem{ES2004}
Ezra Miller and Bernd Sturmfels.
\newblock {\em Combinatorial Commutative Algebra}.
\newblock Springer, 2004.

\bibitem{MT2005}
Marcel Morales and Apostolos Thoma.
\newblock Complete intersection lattice ideals.
\newblock {\em Journal of Algebra}, 284(2):755--770, 2005.

\bibitem{M2019}
Walter D. Morris Jr.
\newblock  Acyclic digraphs giving rise to complete intersections.
\newblock{\em Journal of Commutative Algebra}, 11(2):241--264, 2019.

\bibitem{ohsugi1999koszul}
Hidefumi Ohsugi and Takayuki Hibi.
\newblock Koszul bipartite graphs.
\newblock {\em Advances in Applied Mathematics}, 22(1):25--28, 1999.

\bibitem{ohsugi1999toric}
Hidefumi Ohsugi and Takayuki Hibi.
\newblock Toric ideals generated by quadratic binomials.
\newblock {\em Journal of Algebra}, 218(2):509--527, 1999.

\bibitem{ohsugi2017grobner}
Hidefumi Ohsugi and Takayuki Hibi.
\newblock A Gr{\"o}bner basis characterization for chordal comparability
  graphs.
\newblock {\em European Journal of Combinatorics}, 59:122--128, 2017.

\bibitem{PS2014}
Sonja Petrovi{\'c} and Despina Stasi.
\newblock Toric algebra of hypergraphs.
\newblock {\em Journal of Algebraic Combinatorics}, 39(1):187--208, 2014.

\bibitem{reyes2005complete}
Enrique Reyes.
\newblock Complete intersection toric ideals of oriented graphs.
\newblock {\em Morfismos}, 9(2):71--82, 2005.

\bibitem{RTT2012minimal}
Enrique Reyes, Christos Tatakis, and Apostolos Thoma.
\newblock Minimal generators of toric ideals of graphs.
\newblock {\em Advances in Applied Mathematics}, 48(1):64--78, 2012.

\bibitem{R2017bound}
Kamil Rychlewicz.
\newblock A bound on degrees of primitive elements of toric graph ideals.
\newblock {\em arXiv preprint arXiv:1701.07137}, 2017.

\bibitem{S2002}
Frank Sottile.
\newblock Toric ideals, real toric varieties, and the algebraic moment map.
\newblock {\em arXiv preprint math/0212044}, 2002.

\bibitem{S1996}
Bernd Sturmfels.
\newblock {\em Gr{\"o}bner bases and convex polytopes}, volume~8.
\newblock American Mathematical Soc., 1996.

\bibitem{T2011universal}
Christos Tatakis and Apostolos Thoma.
\newblock On the universal Gr{\"o}bner bases of toric ideals of graphs.
\newblock {\em Journal of Combinatorial Theory, Series A}, 118(5):1540--1548,
  2011.

\bibitem{TT2013}
Christos Tatakis and Apostolos Thoma.
\newblock On complete intersection toric ideals of graphs.
\newblock {\em Journal of Algebraic Combinatorics}, 38(2):351--370, 2013.

\bibitem{TT2015}
Christos Tatakis and Apostolos Thoma.
\newblock Graver degrees are not polynomially bounded by true circuit degrees.
\newblock {\em Journal of Pure and Applied Algebra}, 219(7):2658--2665, 2015.

\bibitem{Ree}
R.~H. Villarreal.
\newblock Rees algebras of edge ideals.
\newblock {\em Communications in Algebra}, 23(9):3513--3524, 1995.
\end{thebibliography}

\setcounter{section}{0}
\renewcommand{\thesection}{\Alph{section}}

\section{Appendix}

\subsection{Proof of Propositions~\ref{rmk:basic:binomial}}

\begin{proof}[Proof of Proposition~\ref{rmk:basic:binomial}]

Recall the definition of a signed multigraph $(G_{\mathbb{b}},\tau_{\mathbb{b}})$,
that is, $G_{\mathbb{b}}$ is induced by $|b_e|$ copies of $e$ for every edge $e$, and
$\tau_{\mathbb{b}}(e,v)=\tau(e',v)$ if $e$ is a copy of $e'\in E(G)$.
For simplicity, we let $H=G_{\mathbb{b}}$.

Let $E_H^{++}(v)$ (resp.  $E_H^{+-}(v)$) be the (multi)set of edges $e$ of $H$ incident to $v$ with $b_e>0$ and $\tau_{\mathbb{b}}(e,v)=1$ (resp. $\tau_{\mathbb{b}}(e,v)=-1$).
Similarly, let $E_H^{-+}(v)$ (resp.  $E_H^{--}(v)$) be the (multi)set of edges $e$ of $H$ incident to $v$ if $b_e<0$ and $\tau_{\mathbb{b}}(e,v)=1$ (resp. $\tau_{\mathbb{b}}(e,v)=-1$).
Note that $\deg_H(v)=|E_H^{++}(v)|+|E_H^{--}(v)|+|E_H^{+-}(v)|+|E_H^{-+}(v)|$, where the size $|M|$ of a multiset $M$ counts multiplicity.

\begin{claim}\label{eq:H}
For every vertex $v$ of $G$,\[|E_H^{++}(v)|+|E_H^{--}(v)|=|E_H^{+-}(v)|+|E_H^{-+}(v)|.\]
\end{claim}
\begin{proof}[Proof of Claim~\ref{eq:H}]
Let $v$ be a vertex  in $H$. The entry of $A\mathbb{b}$ corresponding to $v$
is $\sum_{e\in E_G(v)}b_e\tau(e,v)=0$, where $E_G(v)$ denotes the set of all edges of $G$ incident to $v$. Note that
\begin{eqnarray*}&&\sum_{e\in E_G(v)}b_e\tau(e,v)=0   \    \
 \Leftrightarrow    \  \ \sum_{\substack{e\in E_G(v)\\ \tau(e,v)=1}}b_e = \sum_{\substack{e\in E_G(v) \\ \tau(e,v)=-1}}b_e, \end{eqnarray*}
 and by definition,
\begin{eqnarray*}&&\sum_{\substack{e\in E_G(v)\\ \tau(e,v)=1}}b_e= |E_H^{++}(v)|-|E_H^{-+}(v)|,
\qquad \qquad \sum_{\substack{e\in E_G(v) \\ \tau(e,v)=-1}}b_e=  |E_H^{+-}(v)|-|E_H^{--}(v)|.
 \end{eqnarray*}
Thus,
$|E_H^{++}(v)|-|E_H^{-+}(v)|=|E_H^{+-}(v)|-|E_H^{--}(v)|$, and so the claim holds.
\end{proof}
For simplicity, let
\begin{eqnarray*}&&
E_H^+(v) = E_H^{++}(v) \cup E_H^{--}(v), \qquad \qquad  E_H^-(v) = E_H^{+-}(v)\cup E_H^{-+}(v).
 \end{eqnarray*}
Take any edge $e_{j_1}\in E(H)$ from a nontrivial connected component $D$ of $H$, say $v_{i_1}$ and $v_{i_2}$ are the endpoints, and let $\bw_1$ be the walk $v_{i_1}e_{j_1}v_{i_2}$.
Without loss of generality, we may assume $e_{j_1}\in E_H^+(v_{i_2}) $ (other cases are similar).
By Claim~\ref{eq:H}, we can take an edge $e_{j_2}\in E_H^-(v_{i_2})$, say the endpoint of $e_{j_2}$ other than $v_{i_2}$ is $v_{i_3}$, and then we have a walk $\bw_2: v_{i_1}e_{j_1}v_{i_2}e_{j_2}v_{i_3}$ so that two edge terms incident to $v_{i_2}$ belong to
$E_H^+(v_{i_2}) $ and $E_H^-(v_{i_2})$, respectively.
We choose a walk repeatedly by a same way. To be precise,
suppose that a walk $\bw_{\ell}: v_{i_1}e_{j_1}v_{i_2} \cdots v_{i_{\ell}} e_{j_{\ell}} v_{i_{\ell+1}}$ is selected. Then repeat the following process ($\S$) until no more edge can be selected.
\begin{itemize}
  \item[($\S$)] If $e_{j_{\ell}}\in E_H^+(v_{i_{\ell+1}})$, then we choose an edge $e_{j_{\ell+1}}\in  E_H^-(v_{i_{\ell+1}})-\{e_{j_1},\ldots,e_{j_{\ell}}\}$ (as long as it is not empty), and
if $e_{j_{\ell}}\in E_H^-(v_{i_{\ell+1}})$, then we choose an edge $e_{j_{\ell+1}}\in  E_H^+(v_{i_{\ell+1}})-\{e_{j_1},\ldots,e_{j_{\ell}}\}$ (as long as it is not empty), and then let $v_{i_{\ell+2}}$ be the other endpoint of $e_{j_{\ell+1}}$ and set $\bw_{\ell+1}: v_{i_1}e_{j_1}v_{i_2} \cdots v_{i_{\ell}} e_{j_{\ell}} v_{i_{\ell+1}}e_{j_{\ell+1}} v_{i_{\ell+2}}$.
\end{itemize}
Let $\bw_D$ be the walk lastly obtained.
Since every edge of $D$ is selected at most once at each step in the process, the length of $\bw_D$ is bounded by the number of edges in $D$. We choose such $\bw_D$ as long as possible (maximizing its length).

\begin{claim}\label{claim:even}
For each connected component $D$ of $H$,
the walk $\bw_D$ is an Eulerian of $D$ and it is an even-signed closed walk in $(G,\tau)$.
\end{claim}
\begin{proof}
For simplicity, we denote $\bw_D$ by $\bw$, and let $\bw:v_{i_1}e_{j_1}\cdots e_{j_t}v_{i_{t+1}}$.
Without loss of generality, we assume that $e_{j_t}\in E_H^+(v_{i_{t+1}})$.
First, we show that $\bw$ is closed.
Suppose that $v_{i_1}\neq v_{i_{t+1}}$. Let $I$ be the set of indices $\ell$ for the vertices $v_{i_{\ell}}$ of $\bw$ such that $v_{i_{\ell}}=v_{i_{t+1}}$. Clearly, $t+1\in I$ and $1\not\in I$.
By the way of choosing the walk, for each $\ell\in I\setminus\{t+1\}$, one of $e_{j_{\ell-1}}$ and $e_{j_{\ell}}$ counts 1 of one of the sets $E_H^+(v_{i_{\ell}})$ and $E_H^-(v_{i_{\ell}})$, and the other edge counts 1 of the other set.
Hence,
\[  |E_H^+(v_{i_{t+1}}) \cap\{e_{j_1},\ldots,e_{j_{t-1}}\}| =|E_H^-(v_{i_{t+1}}) \cap\{e_{j_1},\ldots,e_{j_{t-1}}\}|.\]
Since $e_{j_{t}}\in E_H^+(v_{i_{t+1}})$, by Claim~\ref{eq:H}, it follows that
$E_H^-(v_{i_{t+1}})-\{e_{j_1},\ldots,e_{j_t}\}\neq\emptyset$. Thus, we reach a contradiction that $\bw_t$ is a longest one. Thus, $v_{i_1}=v_{i_{t+1}}$, which means $\bw$ is a closed walk.

Suppose that there is an edge of $D$ not covered by the closed walk $\bw$.
Let $D'$ be the graph obtained from $D$ be deleting the edges of $[\bw]$.
Then by the choice of $\bw$, it follows that
\[  \forall v\in  V(D), \quad |E^+_D(v)|=|E^-_D(v)|.\]
By taking a nontrivial connected component of $D'$, we can proceed the same argument in $(\S)$ to obtain a closed walk $\bw'$.
Since both $\bw$ and $\bw'$ are closed, we may assume that both walks start at the vertex $v_{i_1}$. Then the closed walk $\bw+\bw'$ is a longer closed walk which can be obtained from the procedure $(\S)$, a contradiction.
Hence, $\bw$ is an Eulerian of $D$.

It remains to show that $\bw$ is an even-signed closed walk in $(G,\tau)$.
From the definition, it is clear that $\bw$ is a closed walk in $G$, and so it is sufficient to show that $\bw$ has an even number of unbalanced vertex terms.
Note that if the $\ell$th vertex term $v_{i_{\ell}}$ is unbalanced, then
for the three consecutive terms  $e_{j_{\ell-1}}v_{i_{\ell}}e_{j_{\ell}}$ of $\bw$,
one of the four holds: (1) $e_{j_{\ell-1}}\in E_H^{++}(v_{i_{\ell}})$ and $e_{j_{\ell}}\in E_H^{-+}(v_{i_{\ell}})$;
(2) $e_{j_{\ell}}\in E_H^{++}(v_{i_{\ell}})$ and $e_{j_{\ell-1}}\in E_H^{-+}(v_{i_{\ell}})$;
(3) $e_{j_{\ell-1}}\in E_H^{--}(v_{i_{\ell}})$ and $e_{j_{\ell}}\in E_H^{+-}(v_{i_{\ell}})$;
(4) $e_{j_{\ell}}\in E_H^{--}(v_{i_{\ell}})$ and $e_{j_{\ell-1}}\in E_H^{+-}(v_{i_{\ell}})$.
Then the number of unbalanced vertex terms of $\bw$ is
\[\sum_{v\in V(D)} \left( \left||E_H^{++}(v)|-|E_H^{+-}(v)|\right|  +  \left||E_H^{--}(v)|-|E_H^{-+}(v)|\right| \right) =
\sum_{v\in V(D)} 2 \left||E_H^{++}(v)|-|E_H^{+-}(v)|\right|
,\]
where the equality is from Claim~\ref{eq:H}.
Hence, $\bw$ is even-signed.
\end{proof}

By Claim~\ref{claim:even}, we consider a walk $\bw_D:v_{i_1}e_{j_1}v_{i_2}\cdots v_{i_t}e_{j_t}v_{i_{t+1}}$ and its associated binomial $B_{\bw_D}$ for a fixed connected component $D$ of $H$.
To complete the proof, it is sufficient to show that
\begin{eqnarray*}\label{eq:H:last}
  && \mathbb{e}^{\mathbb{b}^+} = \prod_{\substack{D:\text{ connected }\\ \text{component of }H}} B_{\bw_D}^+,\qquad\text{and}\qquad  \mathbb{e}^{\mathbb{b}^-}
  =\prod_{\substack{D:\text{ connected }\\ \text{component of }H}} B_{\bw_D}^-.
\end{eqnarray*}

\begin{claim}\label{claim:sign}
For any $\ell\in [t]$,
$b_{e_{j_{\ell-1}}} b_{e_{j_{\ell}}}<0$  if and only if the $\ell$th vertex term $v_{i_{\ell}}$ is unbalanced.
\end{claim}
\begin{proof}[Proof of Claim~\ref{claim:sign}]
Without loss of generality, we may assume that $b_{e_{j_{\ell}}}>0$. First, we suppose that $b_{e_{j_{\ell-1}}}>0$. Then $e_{j_{\ell-1}}\in E_H^{+-}(v_{i_{\ell}})\cup E_H^{++}(v_{i_{\ell}})$. More precisely, if $e_{j_{\ell-1}}\in E_H^{+-}(v_{i_{\ell}})$ then $e_{j_{\ell}}\in E_H^{++}(v_{i_{\ell}})$, and if $e_{j_{\ell-1}}\in E_H^{++}(v_{i_{\ell}})$ then $e_{j_{\ell}}\in E_H^{+-}(v_{i_{\ell}})$.
Thus, $\tau(e_{j_{\ell-1}},v_{i_{\ell}})\tau(e_{j_{\ell}},v_{i_{\ell}})=-1$, which implies that the vertex term $v_{i_{\ell}}$ is not unbalanced. On the other hand, if we suppose that $b_{e_{j_{\ell-1}}}<0$, then $e_{j_{\ell-1}}\in E_H^{-+}(v_{i_{\ell}})\cup  E_H^{--}(v_{i_{\ell}})$. More precisely, if $e_{j_{\ell-1}}\in E_H^{-+}(v_{i_{\ell}})$ then $e_{j_{\ell}}\in E_H^{++}(v_{i_{\ell}})$, and if $e_{j_{\ell-1}}\in E_H^{--}(v_{i_{\ell}})$ then $e_{j_{\ell}}\in E_H^{+-}(v_{i_{\ell}})$. Thus, $\tau(e_{j_{\ell-1}},v_{i_{\ell}})\tau(e_{j_{\ell}},v_{i_{\ell}})=1$, which implies that $v_{i_{\ell}}$ is unbalanced.
\end{proof}

Take any edge $e$ of $H$. We assume that $b_e>0$, and the other case is similar.
Let $D$ be the connected component of $H$ containing the edge $e$.
We will show that the power of $e$ in
$B_{\bw_D}^+$ is equal to $b_e$.
If there is no unbalanced vertex term in ${\bw_D}$, then $B_{\bw_D}=B^+_{\bw_D}-1$ and so clearly
it holds.
Suppose that $\bw_D$ has an unbalanced vertex term.
Then let $\bw_D=\bw_0+\cdots+\bw_{2k-1}$ be a balanced section-decomposition of $\bw_D$ so that the first edge term $e_{j_1}$ is in $\bw_0$ if $b_{e_{j_1}}>0$ and
$e_{j_1}$ is in $\bw_1$ if $b_{e_{j_1}}<0$.
By Claim~\ref{claim:sign}, for every edge term $e_{j_{\ell}}$ of $\bw_D$, $e_{j_{\ell}}$ is in some $(2s)$th section if $b_{e_{j_{\ell}}}>0$,  and  $e_{j_{\ell}}$ is in some $(2s+1)$th section if $b_{e_{j_{\ell}}}<0$, which completes the proof.
\end{proof}

\subsection{Proof of Proposition~\ref{prop:rank}}

\begin{lemma}\label{rmk:odd:cycle}
For an odd-signed closed walk $\bw$,
$[\bw]$ contains a cycle of $G$, which is odd-signed in $(G,\tau)$. \end{lemma}
\begin{proof}
We show it by the induction on the length of the walk.
If it has length at most three, then it is trivial.
Suppose that the lemma holds for any odd-signed walk of length less than $\ell$ ($\ell>3$). Let $\bw$ be an odd-signed walk of length $(\ell+1)$.
If there is no repeated vertex in $\bw$, then $[\bw]$ is a cycle.
Suppose that there is a repeated vertex in $\bw$.
If $v$ is repeated in $\bw$, then we let $\bw=\bw_0+\bw_1$ so that each $\bw_{i}$ is a closed nontrivial walk with first vertex term $v$, and then we have $-1=\mu(\bw)=\mu(\bw_0)\mu(\bw_1)$ by Lemma~\ref{eq:sign:product:closed}, which implies that one of $\bw_0$ and $\bw_1$ is an odd-signed closed walk, say $\bw_0$.
Since $\bw_0$ is a proper subwalk of $\bw$, by the induction hypothesis, $[\bw_0]$ contains a cycle which is  odd-signed in $(G,\tau)$ and so $[\bw]$ does.
\end{proof}

\begin{proof}[Proof of Proposition~\ref{prop:rank}] Let $A=A(G,\tau)$ and $A(e)$ be the column of $A$ corresponding to an edge $e$. Take a spanning tree $T$ of $G$. Then clearly, the submatrix obtained by the columns corresponding the edges of $T$ has the rank $|V(G)|-1$. Thus,
$\mathrm{rank}(A)\ge |V(G)|-1$.

Suppose to contrary that $(G,\tau)$ contains no odd-signed closed walk and there are $|V(G)|$ linearly independent columns. Let $A'$ be the submatrix inducted by those columns.
Then the subgraph of $G$ induced by the edges corresponding to  the columns of $A'$ has $|V(G)|$ edges and so it contains a cycle $C$.
By the assumption that $(G,\tau)$ has no odd-signed cycle, $C$ is an even-signed cycle and we let $C$ have a balanced section-decomposition $\bw_0+\cdots+\bw_{r}$.
Without loss of generality, we assume that $\bw_i:v^{i}_{1} e^i_{1} \cdots v^i_{a_i} e^i_{a_i} v^i_{a_i+1}$ for each $i$.
If $r=0$, then
it is easy to see that $\sum_{e\in E(C)} A(e)=\mathbb{0}$, a contradiction to the fact that
 $\{A(e)\mid e\in E(C)\}$ are linearly independent.
Suppose that $r>0$. Then each $v^i_1$ is unbalanced.
Since $v^i_{a_i+1}=v^{i+1}_{1}$,
$ \tau(e^i_1,v^i_{a_i+1})=\tau(e^{i+1}_{1},v^{i+1}_{1})$.
Then it follows that \[\sum_{i=0}^{r} (-1)^{i}\sum_{j=1}^{a_i} A(e^i_j)=\mathbb{0}.\]
This is a contradiction to the fact that
 $\{A(e)\mid e\in E(C)\}$ are linearly independent.

Suppose that $(G,\tau)$ contains an odd-signed closed walk $\bw$, and then  $(G,\tau)$ contains an odd-signed cycle $C$ by Lemma~\ref{rmk:odd:cycle}.
We take a unicyclic spanning subgraph $H$ of $G$, containing the cycle $C$. Let $A'$ be the submatrix of $A$ induced by the columns corresponding to the edges of $H$. Note that $A'$ is an $|V(G)|\times |V(G)|$ matrix and by permuting lines, we may assume that
$A'=\left[\begin{array}{cc}
A_{11} & A_{12} \\
O & A_{22} \\
\end{array}
\right]
$ where $A_{ii}$'s are square matrices, the rows and the columns corresponding to $A_{11}$ are the vertices and the edges of the cycle $C$,
 and $A_{22}$ is a upper triangular matrix without zero diagonal element.
Since $\det(A')=\det(A_{11})\det(A_{22})$ and $\det(A_{22})\neq 0$, it is sufficient to show that the columns of the submatrix $A_{11}$ are linearly independent.

Let $C:v_1e_1v_2\ldots v_re_rv_1$.
Suppose that $\sum_{i=1}^{r} c_iA(e_i)=\mathbb{0}$ for some constant $c_i$.
Since the entry of $\sum_{i=1}^{r} c_iA(e_i)$ corresponding to a vertex $v_j$ ($j\in[r]$) is equal to $c_{j-1} \tau(e_{j-1},v_j) + c_{j} \tau(e_{j},v_j)$, where the indices are modulo $r$,
we have $c_{j-1} \tau(e_{j-1},v_j) + c_{j} \tau(e_{j},v_j)=0$.
If a vertex term $v_j$ of $C$ is not unbalanced
then $c_{j-1}=c_j$, and if a vertex term $v_i$ is unbalanced then $c_{j-1}=-c_j$.
Since $C$ has an odd number of unbalanced vertex terms, then it follows that $c_1=c_2=\cdots=c_r=0$.
\end{proof}

\subsection{Proof of Proposition~\ref{lem:that_lemma}}
\begin{figure}[h!]
 \centering
 \includegraphics[width=16cm,page=26]{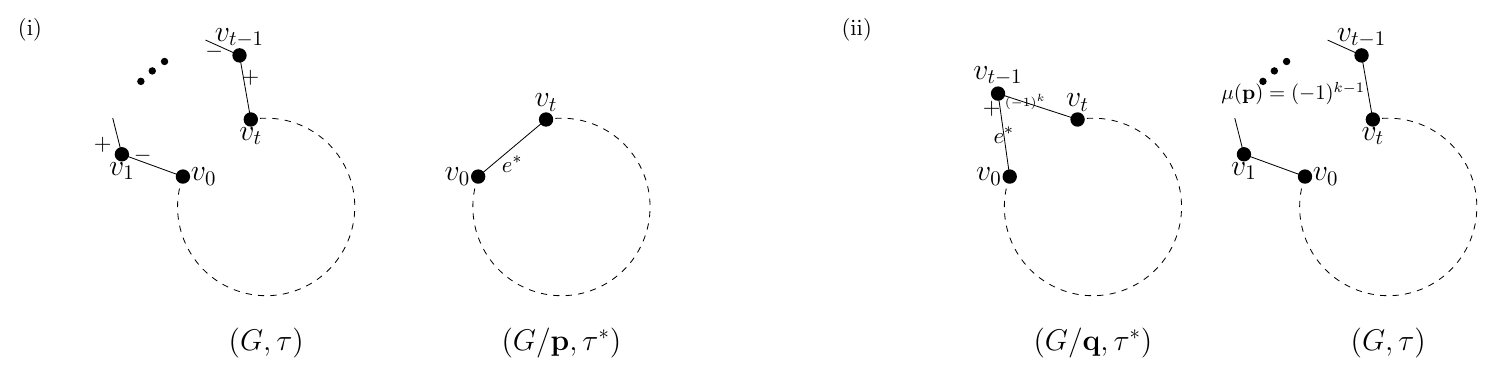}\\
  \caption{An illustration for the proof of Proposition~\ref{lem:that_lemma}}
  \label{fig:ear}
\end{figure}

\begin{proof}[Proof of Proposition~\ref{lem:that_lemma}] Note that $G/\mathbf{p}$ or $G/\mathbf{q}$ in the cases is a simple graph.
For every sign $\tau$ of $G$, every primitive walk $\bw$ in $(G,\tau)$ containing an edge of the path $\mathbf{p}$ contains all edges of $\mathbf{p}$ by Theorem~\ref{thm:signed:primitive}.
In the following, let $G^*$ be $G/\mathbf{p}$ or $G/\mathbf{q}$ (according to the cases), and let $e^*$ be its newly added edge.
To show (i), suppose that $G\in \mathcal{G}^{cis}$. Take a sign $\tau^*$ of $G^*$, and define a sign $\tau$ of $G$ as follows, and see Figure~\ref{fig:ear}.
\[\tau(e,y)=
\begin{cases}
\tau^*(e,y) &\text{if }e\in E(G)\setminus E(\mathbf{p}),\\
\tau^*(e^*,v_0)&\text{if }(e,y)=(v_0v_1,v_0),\\
\tau^*(e^*,v_t)&\text{if }(e,y)=(v_{t-1}v_{t},v_t),\\
1&\text{if }(e,y)=(v_iv_{i+1},v_i) \text{ for some }i\in[t-1],\\
-1&\text{if }(e,y)=(v_{i-1}v_{i},v_i) \text{ for some }i\in[t-1].
\end{cases}\]
For a closed walk $\bw$,
$\bw$ is even-signed in $(G,\tau)$ if and only if the walk $\bw^*$ obtained from $\bw$ by contracting the ear $\mathbf{p}$ is even-signed in $(G^*,\tau^*)$.
Thus, $r(G,\tau)=r(G^*,\tau^*)$ and there are $s$ primitive binomials of $I_{(G,\tau)}$ generating  $I_{(G,\tau)}$ if and only if there are $s$ primitive binomials of $I_{(G^*,\tau^*)}$ generating  $I_{(G^*,\tau^*)}$.
Therefore, $G/\mathbf{p}\in \mathcal{G}^{cis}$ since $G\in\mathcal{G}^{cis}$.

To show (ii), take a sign $\tau$ of $G$.
Consider the balanced section-decomposition of $\mathbf{p}$, and let $k$ be the number of balanced sections of $\mathbf{p}$.
Define a sign $\tau^*$ of $G^*$ as follows, and see Figure~\ref{fig:ear}.
\[\tau^*(e,y)=
 \begin{cases}
\tau(e,y) &\text{if }e\not\in\{e^*,v_{t-1}v_t\}, \\
\tau(v_0v_1,v_0)&\text{if }(e,y)=(e^*,v_0),\\
\tau(v_{t-1}v_t,v_t)&\text{if }(e,y)=(v_{t-1}v_{t},v_t),\\
1&\text{if }(e,y)=(e^*,v_{t-1}), \\
(-1)^k &\text{if }(e,y)=(v_{t-1}v_{t},v_{t-1}).
\end{cases}\]
For a closed walk $\bw$,
$\bw$ is even-signed in $(G,\tau)$ if and only if the walk $\bw^*$ obtained from $\bw$ by contracting the ear $\mathbf{q}$ is even-signed in $(G^*,\tau^*)$.
Similar to the argument of (i),  we have  $G\in\mathcal{G}^{cis}$ since $G/\mathbf{q}\in \mathcal{G}^{cis}$.
\end{proof}

\end{document}